\documentclass[12pt]{elsarticle}
\usepackage{graphicx}
\usepackage{amssymb}
\usepackage{amsmath}
\usepackage{amsfonts}
\usepackage{fullpage}

\newtheorem{theorem}{Theorem}
\newtheorem{assumption}{Assumption}
\newtheorem{algorithm}{Algorithm}
\newtheorem{corollary}{Corollary}
\newtheorem{lemma}{Lemma}
\newdefinition{definition}{Definition}
\newtheorem{proposition}{Proposition}
\newenvironment{proof}[1][Proof]{\noindent\textbf{#1.} }{\ \rule{0.5em}{0.5em}}
\renewenvironment{frontmatter}{}{}
\input{tcilatex}

\begin{document}

\begin{frontmatter}
\title{Adaptive BDDC in Three Dimensions}
\author[ucd]{Jan Mandel}
\ead{jan.mandel@ucdenver.edu}
\author[ucd,ut]{Bed\v{r}ich Soused\'{\i}k}
\ead{bedrich.sousedik@ucdenver.edu}
\author[mu,fs]{Jakub \v{S}\'{\i}stek\corref{cor1}}
\ead{sistek@math.cas.cz}
\cortext[cor1]{Corresponding author; tel.: +420 222 090 710; fax: +420 222 211 638.}
\address[ucd]{Department of Mathematical and Statistical Sciences,
University of Colorado Denver,\\ Denver, CO 80217-3364, USA.}
\address[ut]{Institute of Thermomechanics, Academy of Sciences of the Czech
Republic,\\ Dolej\v{s}kova 1402/5, CZ - 182 00 Prague 8, Czech Republic.}
\address[mu]{Institute of Mathematics, Academy of Sciences of the
Czech Republic, \\ \v Zitn\' a 25, CZ - 115 67 Prague 1, Czech Republic.}
\address[fs]{Department of Mathematics, Faculty of Mechanical Engineering, \\
Czech Technical University in Prague,  \\ Karlovo n\'{a}m\v{e}st\'{\i} 13, CZ -
121 35 Prague 2, Czech Republic.}
\begin{abstract}
The adaptive BDDC method is extended to the selection of face constraints in three dimensions.
A new implementation of the BDDC method is presented based on a global formulation without
an explicit coarse problem, with massive parallelism provided by a multifrontal solver.
Constraints are implemented by a projection and
sparsity of the projected operator is preserved by a generalized change of variables.
The effectiveness of the method is illustrated on several engineering problems.
\end{abstract}
\maketitle
\end{frontmatter}

\section{Introduction}

The \emph{Balancing Domain Decomposition by Constraints} (BDDC) was
developed by Dohrmann~\cite{Dohrmann-2003-PSC} as a~primal alternative to
the \emph{Finite Element Tearing and Interconnecting - Dual, Primal}
(FETI-DP) by Farhat et al.~\cite{Farhat-2000-SDP}. Both methods use
constraints to impose equality of new \textquotedblleft
coarse\textquotedblright\ variables on substructure interfaces, such as
values at substructure corners or weighted averages over edges and faces.
Primal variants of the FETI-DP were also independently proposed by Cros~\cite%
{Cros-2003-PSC} and by Fragakis and Papadrakakis~\cite{Fragakis-2003-MHP}.
It has been shown in~\cite{Mandel-2007-BFM, Sousedik-2008-EPD} that these
methods are in fact the same as BDDC. Polylogarithmic condition number
bounds for FETI-DP were first proved in~\cite{Mandel-2001-CDP} and
generalized to the case of coefficient jumps between substructures in \cite%
{Klawonn-2002-DPF}. The same bounds were obtained for BDDC in~\cite%
{Mandel-2003-CBD,Mandel-2005-ATP}. A proof that the eigenvalues of the
preconditioned operators of both methods are actually the same except for
the eigenvalues equal to one was given in~\cite{Mandel-2005-ATP} and then
simplified in~\cite{Brenner-2007-BFW, Li-2006-FBB,Mandel-2007-BFM}. FETI-DP,
and, equivalently,\ BDDC are quite robust. It can be proved that the
condition number remains bounded even for large classes of subdomains with
rough interfaces in 2D \cite{Klawonn-2008-AFA,Widlund-2008-AIS} as well as
in many cases of strong discontinuities of coefficients, including some
configurations when the discontinuities cross substructure boundaries 
\cite{Pechstein-2008-AFM,Pechstein-2009-AFM}. However, the condition number 
deteriorates in many situations of practical importance and 
a better selection of constraints is desirable.
Enriching the coarse space so that the iterations run
in a subspace devoid of ``difficult'' modes
has been a successful trick in iterative substructuring methods used,
e.g., in the development of BDD and FETI\ for plates from the base BDD and
FETI methods 
\cite{Farhat-1998-TFM1,LeTallec-1994-BDD,LeTallec-1998-NND,Mandel-1999-SSM}.
Methods that build a coarse space adaptively from local eigenvalue
calculations were also devised in other (though related) contexts 
\cite{Brezina-1999-IMC,Fish-1997-GAM,Mandel-1993-IBI,Mandel-1994-IPF,Poole-2003-AAC}. 
Adaptive enrichment for BDDC and FETI-DP was proposed in \cite{Mandel-2006-ACS,Mandel-2007-ASF},
with the added coarse functions built
from eigenproblems based on adjacent pairs of substructures in 2D. The
adaptive method, however, was formulated in terms of FETI-DP operators, and
it was quite complicated.

Here, we develop the adaptive algorithm directly in terms of BDDC operators,
resulting in a much simpler formulation and implementation. Of course,
the algorithm still allows a~translation into the language of the FETI-DP.
We then extend the construction from \cite{Mandel-2006-ACS,Mandel-2007-ASF}
to 3D. We find that the heuristic eigenvalue-based estimates still work
reasonably well and that our adaptive approach can result in the
concentration of computational work in a small troublesome part of the
problem, which leads to a good convergence behaviour at a small added cost.

We also develop a new implementation framework that operates on global
matrices, builds no explicit coarse problem, and gets much of its
parallelism through the direct solver used for solution of an auxiliary
decoupled system. To preserve sparsity, we use a variant of the change of
variables from~\cite{Li-2006-FBB}, extended to an arbitrary number of
constraints. Our current parallel implementation is built on top of the
multifrontal massively parallel sparse direct solver MUMPS~\cite%
{Amestoy-2000-MPD}, motivated also by an earlier implementation of the BDDC
preconditioner based on the frontal solver~\cite{Sistek-2010-BFS}.

The rest of the paper is organised as follows. In Section~\ref{sec:BDDC}, we
establish the notation and review the BDDC algorithm in a form suitable for
our purposes. In Section \ref{sec:adaptive-constraints}, we describe the
adaptive method. Section~\ref{sec:mumps} then describes the implementation
on top of a massively parallel direct solver. Section~\ref{sec:generalized}
presents the generalized change of variables to preserve sparsity. Section~%
\ref{sec:implement} describes some further details of the implementation.
Numerical results are presented in Section~\ref{sec:numerical}. Section~\ref%
{sec:conclusion} contains the summary and concluding remarks.

Some of results in this paper were presented in the thesis \cite%
{Sousedik-2008-CDD}.

\section{Notation, substructuring, and BDDC}

\label{sec:BDDC}

To establish notation, we first briefly review standard substructuring
concepts and state the BDDC method in a form suitable for our purposes. The
setting and notation here is compatible with \cite{Mandel-2008-MMB}, with
some additions. See, e.g.,~\cite{Smith-1996-DD,Toselli-2005-DDM} for more
details about iterative substructuring and \cite%
{Kruis-2006-DDM,Mandel-2003-CBD,Mandel-2008-MMB,Smith-1996-DD,Toselli-2005-DDM}
for BDDC.

Consider an elliptic boundary value problem defined on a bounded domain 
$\Omega\subset\mathbb{R}^{3}$ and discretized by
conforming finite elements. The domain $\Omega$ is decomposed into $N$
nonoverlapping \emph{subdomains} $\Omega_{i}$, $i=1,\dots N$, also called
\emph{substructures}, so that each substructure $\Omega_{i}$ is a union of
finite elements. Each node is associated with one degree of freedom in the
scalar case, and with 3 displacement degrees of freedom in the case of
linear elasticity. The nodes contained in more than one substructure 
are called the \emph{interface}, denoted by $\Gamma$, and $\Gamma_{i} = \Gamma \cap \Omega_{i}$ 
is the interface of substructure~$\Omega_{i}$. The interface $\Gamma$ 
may also be classified as the union of three different types of
nonoverlapping sets: \emph{faces}, \emph{edges}, and \emph{corners}. We will adopt here the following 
simple definition. A~\emph{face} contains all nodes shared solely by one pair of subdomains,
an \emph{edge} contains all nodes shared by same
set of more than two subdomains, and a \emph{corner} is a degenerate edge
with only one node; for a more general definition see e.g.~\cite{Klawonn-2006-DPF}. 
Edges and faces are also called \emph{globs}.

We identify finite element functions with the vectors of their coefficients
in the standard finite element basis. These coefficients are also called
\emph{variables} or \emph{degrees of freedom}. We also identify linear
operators with their matrices, in bases that will be clear from the context.

The space of all (vectors of the degrees of freedom of) finite element
functions on subdomain~$\Omega_{i}$ is denoted by $W_{i}$, and let%
\begin{equation}
W=W_{1}\times\cdots\times W_{N}.  \label{eq:W}
\end{equation}
The space $W$ is equipped with the standard $\mathbb{R}^{n}$ basis and
the Euclidean inner product $\left\langle w,v\right\rangle =w^{\mathrm{T}}v$%
. For a symmetric positive semidefinite matrix~$M$, $\left\langle
u,v\right\rangle_{M}=\left\langle Mu,v\right\rangle $, and $\left\Vert
u\right\Vert _{M}=\left\langle Mu,u\right\rangle ^{1/2}$.

Let $A_{i}:W_{i}\rightarrow W_{i}$ be the local substructure stiffness
matrix, obtained by the subassembly of element matrices only in substructure
$\Omega_{i}$. The matrices$~A_{i}$ are symmetric positive semidefinite for an elliptic problem. 
We can write vectors and matrices in the block form
\begin{equation}
w=\left[
\begin{array}{c}
w_{1} \\
\vdots \\
w_{N}%
\end{array}
\right] ,\quad w\in W,\quad A=%
\begin{bmatrix}
A_{1} &  &  \\
& \ddots &  \\
&  & A_{N}%
\end{bmatrix}
:W\rightarrow W.  \label{eq:block-operators}
\end{equation}

Now let $U\subset W$ be the space of all functions from $W$ that are
continuous across substructure interfaces. We are interested in solving the
problem
\begin{equation}
u\in U:\quad\left\langle Au,v\right\rangle =\langle f,v\rangle,\quad
\forall\,v\in U\,,  \label{eq:var}
\end{equation}
where $f\in W$ is a given right-hand side. Vectors from $U$ are called
vectors of \emph{global} degrees of freedom, while vectors from $W_{i}$ are
called \emph{local}. The space $U$ is equipped with the basis of $0$-$1$
vectors with one basis vector for each global degree of freedom. The basis
vectors have $1$s in the places where the global degree of freedom coincides
with a local one. The matrix
\begin{equation}
R:U\rightarrow W
\end{equation}
formed from these basis vectors as columns, is the familiar global-to-local
mapping that restricts the global vectors of degrees of freedom to local
degrees of freedom on each $\Omega_{i}$. Thus, $R^{\mathrm{T}}AR$ is the
global stiffness matrix, and (\ref{eq:var}) is equivalent to the assembled
system 
\begin{equation}
\label{eq:assembled}
R^{\mathrm{T}}ARv=R^{\mathrm{T}}f.
\end{equation}
The matrix $R$ is also the
matrix of the canonical embedding $U\subset W$ in the given bases.

Denote by $U_{I}\subset W$ the space of all (vectors of) finite element
functions with nonzero values only in the interiors of substructures $\Omega_{i}$. 
Then $U_{I}\subset U$, and the space $W$ is decomposed as the $A $-orthogonal direct sum
\begin{equation}
W=U_{I}\oplus W_{H},\quad U_{I}\perp_{A}W_{H},
\label{eq:harmonic-decomposition-W}
\end{equation}
where the functions from $W_{H}$ are called \emph{discrete harmonic}. Such
functions are fully determined by values of degrees of freedom at the interface, 
and they have minimal energy on every subdomain. 
Therefore, in a computer implementation, only interface values of discrete harmonic functions 
need to be stored.

The $A$-orthogonal projection onto $U_{I}$ is denoted by
\begin{equation}
P:W\rightarrow U_{I}.
\end{equation}
For $w\in W$, $\left( I-P\right) w$ is the discrete harmonic extension from
the values of $w$ on the substructure boundaries. The evaluation of $Pw$
consists of the solution of $N$ independent Dirichlet problems, one in each
substructure.

The space of all discrete harmonic functions from $W$ that are continuous
at interface is denoted by $\widehat{W}$. We have
\begin{equation}
\widehat{W}=W_{H}\cap U=\left( I-P\right) U,
\end{equation}
and the $A$-orthogonal decomposition
\begin{equation}
U=U_{I}\oplus\widehat{W},\quad U_{I}\perp_{A}\widehat{W}.
\label{eq:harmonic-decomposition-U}
\end{equation}
The solution $v\in U$ of problem (\ref{eq:assembled}) is split as
\begin{equation}
Rv=u+w,\quad u\in U_{I},\text{ }w\in\widehat{W}.
\end{equation}
Solving for the interior component $u\in U_{I}$ decomposes into $N$
independent Dirichlet problems. We are interested in finding the discrete
harmonic component $w\in\widehat{W}$, which is the solution of the reduced
problem
\begin{equation}
w\in\widehat{W}:\quad\left\langle Aw,z\right\rangle =\left\langle
f,z\right\rangle, \quad\forall z\in\widehat{W}.  \label{eq:reduced-problem}
\end{equation}

We further need an averaging operator
\begin{equation}
E:W\rightarrow U.  \label{eq:E}
\end{equation}
The operator $E$ replaces the variables on the interface 
by their averages (arithmetic or weighted) from all adjacent subdomains, and it
preserves variables in the interiors of substructures. The operator $E$
is a projection from $W$ onto $U$. Then the operator
\begin{equation}
\left( I-P\right) E:W\rightarrow\widehat{W}  \label{eq:harm-E}
\end{equation}
is a projection from $W$ onto $\widehat{W}$. Its evaluation consists of
averaging between the substructures, followed by the discrete harmonic
extension from the substructure boundaries. Also, note that%
\begin{equation}
\left( I-\left( I-P\right) E\right) w=\left( I-P\right) \left( I-E\right) w, 
\quad\forall w\in W_{H},  \label{eq:IPE}
\end{equation}
since $Pw=0$ if $w\in W_{H}$.

Proper weights (e.g. proportional to the substructure stiffness) in the
averaging given by $E$ are important for the performance of BDDC (as well as other
iterative substructuring methods) independent of different stiffness of
substructures \cite{Klawonn-2006-DPF,Mandel-2005-ATP}.

The BDDC preconditioner is characterized by a selection of \emph{coarse
degrees of freedom,} such as values at corners and averages over edges or
faces. The action of the BDDC preconditioner is then defined in the space
given by the requirement that the coarse degrees of freedom on adjacent
substructures coincide, which is enforced in the algorithms by \emph{%
constraints}. So, the design of the BDDC preconditioner is characterized by
a selection of an intermediate space $\widetilde{W}$ satisfying these
constraints,%
\begin{equation}
\widehat{W}\subset\widetilde{W}\subset W_{H}.  \label{eq:w-tilde}
\end{equation}
The BDDC then consists of preconditioned conjugate gradients (PCG) applied
to the problem~(\ref{eq:reduced-problem}) with the preconditioner%
\begin{equation}
M_{BDDC}:r\mapsto u=\left( I-P\right) Ew,\quad w\in\widetilde{W}%
:\quad\left\langle Aw,z\right\rangle =\left\langle r,\left( I-P\right)
Ez\right\rangle ,\quad\forall z\in\widetilde{W},  \label{eq:s-BDDC}
\end{equation}
where $r$ is the residual in the PCG method. The following condition number
bound for BDDC will play an essential role in our design of the adaptive
method.

\begin{theorem}[\protect\cite{Mandel-2005-ATP}]
\label{thm:lambda-bound}The eigenvalues of the preconditioned operator of
the BDDC method satisfy $1\leq\lambda\leq \omega_{BDDC}$, where
\begin{equation}
\omega_{BDDC}=\sup_{w\in\widetilde{W}}\frac{\left\Vert \left( I-\left(
I-P\right) E\right) w\right\Vert _{A}^{2}}{\left\Vert w\right\Vert _{A}^{2}}.
\label{eq:omega}
\end{equation}
\end{theorem}

The BDDC enforces the equality of corner coarse degrees of freedom directly by
using the space $W^{c}$, consisting of all functions where the local degrees
of freedom on the substructure corners coincide. Then
\begin{equation}
U\subset W^{c}\subset W.
\end{equation}
Just like $U$, space $W^{c}$ is equipped with a basis consisting of $0$-$1$%
vectors. The basis vector corresponding to a corner degree of freedom has $1$%
s in the places where the global degree of freedom coincides with the
corresponding substructure degree of freedom. The global-to-local matrix%
\begin{equation}
R^{c}:W^{c}\rightarrow W,
\end{equation}
formed from these basis vectors as columns, is the matrix of the canonical
embedding $W^{c}\subset W$, and
\begin{equation}
A^{c}=R^{c\mathrm{T}}AR^{c}.  \label{eq:c-subassembly}
\end{equation}
is the stiffness matrix assembled at the subdomain corners only 
(Figure \ref{fig:virtual_mesh}).

We require that there are sufficiently many corner constraints,
which leads to the following assumption.

\begin{assumption}
\label{ass:pos-def}The matrix $A$ is positive definite on $W^{c}$.
\end{assumption}

Denote by $\widetilde{W}^{c}$ the space all of discrete harmonic functions
in $W^{c}$,  
\begin{equation}
\widetilde{W}^{c}=W^{c}\cap W_{H}.  \label{eq:tildeWc}
\end{equation}
Then
\begin{equation}
\widehat{W}\subset\widetilde{W}\subset\widetilde{W}^{c}\subset W_{H},  \label{eq:spaces}
\end{equation}
and we construct the space $\widetilde{W}$ by enforcing the remaining
constraints weakly by a matrix $D$,
\begin{equation}
\widetilde{W}=\left\{ w\in\widetilde{W}^{c}:Dw=0\right\} .  \label{eq:Wtilde}
\end{equation}
Each row of $D$ defines one constraint. We require that the
constraints are satisfied by all functions that are continuous across the
interfaces,
\begin{equation}
Dw=0,\quad\forall w\in U.  \label{eq:D-compatible}
\end{equation}
Note that (\ref{eq:D-compatible}) implies that $\widehat{W}\subset
\widetilde{W}$, and that the constraints $Dw=0$ involve boundary variables
only. The adaptive algorithm will construct such matrix $D$.

\begin{lemma}
\label{lem:saddle}The BDDC preconditioner (\ref{eq:s-BDDC}) satisfies
\begin{equation}
M_{BDDC}:r\mapsto u=\left( I-P\right) ER^{c}w_{c}.  \label{eq:BDDC-prec}
\end{equation}
For some $\lambda$, it holds
\begin{equation}
\begin{array}{ccccc}
A^{c}w_{c} & + & D^{c\mathrm{T}}\lambda & = & R^{c\mathrm{T}}E^{\mathrm{T}%
}\left( I-P\right) ^{\mathrm{T}}r, \\
D^{c}w_{c} &  &  & = & 0,
\end{array}
\label{eq:BDDC-system}
\end{equation}
where
\begin{equation}
D^{c}=DR^{c}  \label{eq:BDDC-Dc}
\end{equation}
differs from $D$ only by omitting some zero columns corresponding to corners.
\end{lemma}

\begin{proof}
The saddle point problem (\ref{eq:BDDC-system}) is equivalent to the
constrained minimization
\begin{equation}
\frac{1}{2}\left\langle Aw,w\right\rangle -\left\langle r,\left( I-P\right)
Ew\right\rangle \rightarrow\min\quad\text{subject to }w\in W^{c},Dw=0,
\label{eq:BDDC-min}
\end{equation}
with $w=R^{c}w_{c}$. Let $w$ be a solution of (\ref{eq:BDDC-min}) and $z\in
U_{I}$. Since $w$ is optimal with respect to variation $z$, and $\left(
I-P\right) z=0$, we have $\left\langle Aw,z\right\rangle =0$. Thus, $w$ is
discrete harmonic. It follows that (\ref{eq:BDDC-min}) is equivalent to
\begin{equation}
\frac{1}{2}\left\langle Aw,w\right\rangle -\left\langle r,\left( I-P\right)
Ew\right\rangle \rightarrow\min\quad\text{subject to }w\in\widetilde{W}%
^{c},Dw=0,
\end{equation}
which, by (\ref{eq:Wtilde}), is the same as (\ref{eq:s-BDDC}).

Finally, matrix $DR^{c}$ differs from $D$ only by omitting zero columns
because the constraints $Dw=0$ do not involve corners.
\end{proof}

\begin{remark}
In practice, the computation of $\left( I-P\right) ^{\mathrm{T}}r$ can be
omitted, because $r=Ae$, where the error $e$ is discrete harmonic, and then
\begin{equation}
\left\langle P^{\mathrm{T}}r,z\right\rangle =\left\langle P^{\mathrm{T}
}Ae,z\right\rangle =\left\langle Ae,Pz\right\rangle =\left\langle e,z\right\rangle _{A} =0,\quad\forall z\in U_{I},
\end{equation}
thus $P^{\mathrm{T}}r=0$, that is, $r=0$ in the interiors. The condition
that the error $e$ is discrete harmonic is preserved in the iteration by
induction, and the initial error can be made discrete harmonic by a suitable
choice of initial approximation for the reduced problem (e.g., zero).
\end{remark}

\section{Adaptive selection of constraints}

\label{sec:adaptive-constraints}

We first briefly review the principle of the adaptive method from \cite{Mandel-2007-ASF}, 
in a form suitable for our purposes. The condition number
bound $\omega_{BDDC}$ from Theorem \ref{thm:lambda-bound} equals to the maximum eigenvalue 
$\lambda_{1}$ of the associated generalized eigenvalue problem
\begin{equation}
w\in\widetilde{W}:\quad\left\langle \left( I-\left( I-P\right) E\right)
w,\left( I-\left( I-P\right) E\right) z\right\rangle _{A}=\lambda
\left\langle w,z\right\rangle _{A}, \quad\forall z\in\widetilde{W}.
\label{eq:global-eig-var}
\end{equation}
The following statement is well known from linear algebra, e.g. \cite[Theorem 5.2]{Demmel-1997-ANL}.

\begin{lemma}[Courant-Fisher-Weyl minimax principle]
\label{lem:optimal-eig}
Let $c\left( \cdot,\cdot\right) $ be symmetric
positive semidefinite bilinear form on vector space $V$ of dimension $n$ and
$b\left( \cdot,\cdot\right) $ symmetric positive definite bilinear form on $V$. 
Then the generalized eigenvalue problem
\begin{equation}
w\in V:c\left( w,u\right) =\lambda b\left( w,u\right), \quad\forall u\in V
\end{equation}
has $n$ linearly independent eigenvectors $w_{k}$ and the corresponding
eigenvalues are real and nonnegative and the eigenvectors are stationary
points of the Rayleigh quotient $c\left( w,w\right) /b\left( w,w\right) $,
with the stationary values equal to $\lambda_{i}$. Order $%
\lambda_{1}\geq\lambda_{2}\geq\ldots\geq\lambda_{n}\geq0$. Then, for any
subspace $V_{k}\subset V$ of dimension $n-k$,
\begin{equation*}
\max_{w\in V_{k},w\neq0}\frac{c\left( w,w\right) }{b\left( w,w\right) }%
\geq\lambda_{k+1},
\end{equation*}
with equality if
\begin{equation}
V_{k}=\left\{ w\in V:c(w_{\ell},w)=0,\quad\forall\ell=1,\ldots,k\right\} .
\label{eq:complement}
\end{equation}
\end{lemma}

Since the bilinear form on the left-hand side of (\ref{eq:global-eig-var})
is symmetric positive semidefinite and the bilinear form on the right-hand
side is symmetric positive definite, Lemma \ref{lem:optimal-eig} applies and leads to the 
following corollary.

\begin{corollary}
The generalized eigenvalue problem (\ref{eq:global-eig-var}) has eigenvalues
$\lambda_{1}\geq\lambda_{2}\geq\ldots\geq\lambda_{n}\geq0$. Denote the
corresponding eigenvectors $w_{\ell}$. Then, for any $k=1,\ldots,n-1$, and
any linear functionals $L_{\ell}$ on $W$, $\ell=1,\ldots,k$,%
\begin{equation}
\max\left\{ \frac{\left\Vert \left( I-\left( I-P\right) E\right)
w\right\Vert _{A}^{2}}{\left\Vert w\right\Vert _{A}^{2}}:w\in\widetilde {W},%
\text{ }L_{\ell}\left( w\right) =0,\quad\forall\ell=1,\ldots ,k\right\}
\geq\lambda_{k+1},
\end{equation}
with equality if
\begin{equation}
L_{\ell}\left( w\right) =\left\langle \left( I-\left( I-P\right) E\right)
w_{\ell},\left( I-\left( I-P\right) E\right) w\right\rangle _{A}.
\label{eq:extra-constr-var}
\end{equation}
\end{corollary}

The next lemma shows that the added constraints $L_{\ell}\left( w\right) =0$
satisfy the compatibility condition (\ref{eq:D-compatible}).

\begin{lemma}
\label{lem:D-compatible}The constraints $L_{\ell}\left( w\right) =0$, with $L_{\ell}$ 
given by (\ref{eq:extra-constr-var}), are satisfied for any $w\in U$.
\end{lemma}

\begin{proof}
From (\ref{eq:IPE}), $\left( I-\left( I-P\right) E\right) =\left( I-P\right)
\left( I-E\right) $. For any $w\in U$, $\left( I-E\right) w=0$, because $E$
is a projection on $U$.
\end{proof}

It follows that the optimal decrease of the condition number bound (\ref{eq:omega}) 
can be achieved by adding the rows $d_{\ell}^{\mathrm{T}}$ defined by $d_{\ell}^{\mathrm{T}}
w=L_{\ell}\left( w\right) $ to the constraint matrix $D$ in the definition of $\widetilde{W}$ (\ref{eq:Wtilde}). 

However, solving the global eigenvalue problem
(\ref{eq:global-eig-var}) is expensive, and the vectors $d_{\ell}$ are not
of the form suitable for substructuring, i.e., each $d_{\ell}$ with nonzeros
at one glob only.
For these reasons, we replace (\ref{eq:global-eig-var}) by a collection of
local problems, each defined by considering only two adjacent subdomains 
$\Omega_{i}$ and $\Omega_{j}$ at a~time. Here, subdomains are considered adjacent if
they share an edge in 2D, or a face in 3D (Figure \ref{fig:2domains}). All
quantities associated with such pair will be denoted by the subscript $_{ij}$
. Using also (\ref{eq:IPE}), the generalized eigenvalue problem (\ref{eq:global-eig-var}) becomes
\begin{equation}
w_{ij}\in\widetilde{W}_{ij}:\quad\left\langle \left( I-P_{ij}\right) \left(
I-E_{ij}\right) w_{ij},\left( I-P_{ij}\right) \left( I-E_{ij}\right)
z_{ij}\right\rangle _{A_{ij}}=\lambda\left\langle w_{ij},z_{ij}\right\rangle
_{A_{ij}},\quad\forall z_{ij}\in\widetilde{W}_{ij}.  \label{eq:global-eig-var-ij}
\end{equation}

\begin{assumption}
\label{assum:local-regular} The corner constraints are already sufficient to
prevent relative rigid body motions of any pair of adjacent substructures,
so
\begin{equation}
\forall w_{ij}\in\widetilde{W}_{ij}:{A_{ij}}w_{ij}=0\Rightarrow\left( I-{E}%
_{ij}\right) w_{ij}=0,
\end{equation}
i.e., the corner degrees of freedom are sufficient to constrain the rigid
body modes of the two substructures into a single set of rigid body modes,
which are continuous across the interface~$\Gamma_{ij}$.
\end{assumption}

The maximal eigenvalue $\omega_{ij}$ of (\ref{eq:global-eig-var-ij}) is
finite due to Assumption \ref{assum:local-regular}, and we define the
heuristic \emph{condition number indicator}
\begin{equation}
\widetilde{\omega}=\max\left\{ \omega_{ij}:\Omega_{i}\text{ and }\Omega _{j}%
\text{ are adjacent}\right\} .  \label{eq:cond-ind}
\end{equation}

Considering two adjacent subdomains $\Omega_{i}$ and $\Omega_{j}$ only, we
get the added constraints $L_{\ell}\left( w\right) =0$ from (\ref{eq:extra-constr-var}) as
\begin{equation}
\left\langle \left( I-P_{ij}\right) \left( I-E_{ij}\right) w_{ij,\ell},\left(
I-P_{ij}\right) \left( I-E_{ij}\right) w\right\rangle
_{A_{ij}}=0,\quad\forall\ell=1,\ldots,k_{ij},  \label{eq:extra-constr-var-ij}
\end{equation}
where $w_{ij,\ell}$ are the eigenvectors corresponding to the $k_{ij}$ largest
eigenvalues from (\ref{eq:global-eig-var-ij}).

\begin{algorithm}[Adaptive BDDC]
Find the smallest $k_{ij}$ for each pair of adjacent substructures 
$\Omega_{i}$ and $\Omega_{j}$ to guarantee that
$\lambda_{ij,k_{ij}+1}\leq\tau$, where $\tau$ is a given tolerance, and add the
constraints (\ref{eq:extra-constr-var-ij}) to the definition of $\widetilde{W}$.
\end{algorithm}

The adaptive BDDC method assures that the condition number indicator 
$\widetilde{\omega}\leq\tau$ with the minimum number of added constraints. It
was presented in \cite{Mandel-2007-ASF} starting from corner constraints
only, formulated in terms of FETI-DP, and the result translated to BDDC. We
extend the method to the case a general space $\widetilde{W}$ and give a
much simpler implementation in BDDC directly.

To formulate a numerical algorithm, we need to write the generalized
eigenvalue problem (\ref{eq:global-eig-var-ij}) and the added constraints (\ref{eq:extra-constr-var-ij}) 
in terms of matrices and vectors. 
Consider the space $\widetilde{W}_{ij}$ given by the corner constraints and an initial
constraint matrix $D^{c}_{ij}$ devised from an initial global matrix $D^{c}$. 
Recall that $W^{c}_{ij}$ is the space of functions from $W_{ij}$ that are continuous at corners,
$R^{c}_{ij}:W^{c}_{ij}\rightarrow W_{ij}$ is the identity embedding, 
$A^{c}_{ij}=R^{c\mathrm{T}}_{ij}A_{ij}R^{c}_{ij}$ is the matrix assembled at
the corners, and $D^{c}_{ij}=D_{ij}R^{c}_{ij}$. 
Let 
\begin{equation}
\Pi _{ij}=I-D^{c\mathrm{T}}_{ij}\left( D^{c}_{ij}D^{c\mathrm{T}}_{ij}\right) ^{-1}D^{c}_{ij}
\end{equation}
be the orthogonal projection onto $\limfunc{null}D^{c}_{ij}$,
$\Pi _{ij}: \widetilde{W}^{c}_{ij}\rightarrow \widetilde{W}_{ij}$. 
The initial constraint matrix $D^{c}_{ij}$ can be empty; then $\Pi _{ij}=I$.
The generalized eigenvalue problem (\ref{eq:global-eig-var}) now becomes
\begin{equation}
\Pi _{ij}\left( I-E_{ij}\right) ^{\mathrm{T}}S^{c}_{ij}\left( I-E_{ij}\right) \Pi _{ij}w_{ij}=\lambda _{ij}\Pi _{ij}
S^{c}_{ij}\Pi _{ij}w_{ij},  \label{eq:eig-matrix}
\end{equation}
where
\begin{equation}
S^{c}_{ij}=\left( I-P_{ij}\right) ^{\mathrm{T}}A^{c}_{ij}\left( I-P_{ij}\right) .
\end{equation}
Since
\begin{equation}
\limfunc{null}\Pi _{ij}S^{c}_{ij}\Pi _{ij} \subset \limfunc{null}\Pi _{ij}\left( I-E_{ij}\right)^{\mathrm{T}}
S^{c}_{ij}\left( I-E_{ij}\right) \Pi _{ij},  \label{eq:eig-nullspace}
\end{equation}
the eigenvalue problem (\ref{eq:eig-matrix}) reduces in the factorspace
modulo $\limfunc{null}\Pi _{ij}S^{c}_{ij}\Pi _{ij}$ to a problem with the operator on the
right-hand side positive definite. In some computations, we have used
the subspace iteration method LOBPCG \cite{Knyazev-2001-TOP} to find the
dominant eigenvalues and their eigenvectors. The LOBPCG iterations then
run in the factorspace. To use standard eigenvalue solvers, (\ref{eq:eig-matrix}) 
may be converted to a matrix eigenvalue problem by penalizing
the components in $\limfunc{null}D^{c}_{ij}$ and rigid body modes, as already described in \cite{Mandel-2007-ASF}.

It follows from the matrix form of the eigenvalue problem (\ref{eq:eig-nullspace}), 
that the constraints to be added are
\begin{equation}
L_{ij,\ell }\left( w_{ij}\right) =w_{ij,\ell }^{\mathrm{T}}\Pi _{ij}\left( I-E_{ij}\right) ^{%
\mathrm{T}}S^{c}_{ij}\left( I-E_{ij}\right) \Pi _{ij}w_{ij}=0,
\label{eq:Lijl}
\end{equation}
where $w_{ij}$ is the restriction of $w$ to the pair of subdomains.
That is, we wish to add to the constraint matrix $D$ the rows
\begin{equation}
d_{ij,\ell }=w_{ij,\ell }^{\mathrm{T}}\Pi _{ij}\left( I-E_{ij}\right) ^{\mathrm{T}
}S^{c}_{ij}\left( I-E_{ij}\right) \Pi _{ij}.  \label{eq:d-def}
\end{equation}
These rows are extended to the size of global matrix $D$ simply by zeros.

\begin{proposition}
\label{prop:entry-compatible}
The vectors $d_{ij,\ell }$, constructed for a
domain consisting of only two substructures $\Omega _{i}$ and $\Omega _{j}$,
have matching entries on the interface between the two substructures, with
opposite signs.
\end{proposition}

\begin{proof}
Consider the vector $w\in W$ that has two entries equal to $1$,
corresponding to a~degree of freedom on the interface, and all other entries
equal to $0$. Using the definition of $d_{ij,\ell }$ (\ref{eq:d-def}) and Lemma \ref{lem:D-compatible}, 
we get $d_{ij,\ell }w_{ij}=L_{ij,\ell }\left( w_{ij}\right) =0$.
The proof is concluded by taking $w_{ij}$ of this form for arbitrary degree of freedom,
all of which satisfy relation (\ref{eq:Lijl}).
\end{proof}

In 2D, one can simply add rows (\ref{eq:d-def}) to the constraint matrix $D$, 
which is equivalent to the method from \cite{Mandel-2007-ASF}. In 3D,
unfortunately, such rows would generally have nonzero entries over all
of the interface of $\Omega _{i}$ and $\Omega _{j}$, including the
edges (where $\Omega _{i}$ and $\Omega _{j}$ intersect other substructures).
Consequently, 
these rows would couple several globs together, and
the matrix $D^{c}D^{c\mathrm{T}}$ would be in general no longer block diagonal 
with one block per glob. To preserve the
block diagonal structure, we have to split each $d_{ij,\ell }$ into one row
that contains the nonzero entries of the face, and one row for each edge
that contains the nonzero entries of that edge. From Proposition \ref{prop:entry-compatible}, 
it follows that these split constraints satisfy the
compatibility condition~(\ref{eq:D-compatible}), 
and thus the space $\widetilde{W}$ is well defined.

\begin{remark}
In the computations reported in Section \ref{sec:numerical}, we drop the
adaptively generated edge constraints in 3D. Then it is no
longer guaranteed that the condition number indicator $\widetilde{\omega }%
\leq \tau $. However, the method is still observed to perform well.
\end{remark}

\section{Parallel framework with global matrices and on top of multifrontal solver}
\label{sec:mumps}

The main purpose of BDDC, just like any other iterative
substructuring method, is to split the problem into subproblems,
which are solved independently on separate nodes in a multiprocessor system.
Therefore, the usual implementation results in independent local problems on
the spaces $W_{i}$ and a small coarse problem \cite{Dohrmann-2003-PSC,Mandel-2003-CBD}. 
Parallel implementation then requires a
fair amount of custom coding. To reduce the amount of new code, 
a~BDDC implementation that uses specially
crafted calls to a~frontal solver to compute almost all quantities on the
substructures was developed \cite{Sistek-2010-BFS}. 
However, the frontal solver implementation needs to construct
a coarse problem, and the programmer needs to handle the parallelism explicitly. 
Fortunately, highly
efficient massively parallel direct solvers exist, and an implementation
based on such solver may avoid dealing with parallel issues completely. 

When there are only corner constraints, i.e. $D^{c}$ is empty, the BDDC
preconditioner (\ref{eq:BDDC-prec})--(\ref{eq:BDDC-Dc}) reduces to
\begin{equation}
M_{BDDC}:r\mapsto\left( I-P\right) ER^{c}\left( A^{c}\right) ^{-1}R^{c%
\mathrm{T}}E^{T}\left( I-P\right) ^{\mathrm{T}}r.  \label{eq:BDDC-simple}
\end{equation}
All coupling between substructures in the matrix $A^{c}$ is concentrated
at corner degrees of freedom, while most computational work rests inside
the subdomains, and an efficient solver should be able to perform it
independently in parallel. Our implementation is based on the multi-frontal
solver MUMPS \cite{Amestoy-2000-MPD}, and numerical results show that this solver
can indeed handle matrices of this type reasonably well. Our MATLAB
implementation also uses global matrices. 
In both implementations, expressions involving sparse matrices are evaluated using vectors 
in the space $\widetilde{W}^c$ just as in the formulas here.

However, if there are any constraints in the globs, one has to solve the
constrained system (\ref{eq:BDDC-system}), and MUMPS cannot do this
directly. Thus, we will transform (\ref{eq:BDDC-system}) to a symmetric,
positive definite system which can be solved by MUMPS.

One way to solve system (\ref{eq:BDDC-system}) is to introduce the
orthogonal projection $\Pi$ onto the nullspace of $D^{c}$, which is given by
\begin{equation}
\Pi=I-D^{c\mathrm{T}}\left( D^{c}D^{c\mathrm{T}}\right) ^{-1}D^{c}.
\label{eq:projection}
\end{equation}
As opposed to the localised analogue $\Pi _{ij}$ used in the previous section,
$\Pi :\widetilde{W}^c \rightarrow \widetilde{W}$ is a~global operator.
Due to the block structure of $D^{c}$, where each block corresponds to a~different glob,
and because each degree of freedom belongs to at most one glob by definition, 
the construction of $\Pi$ can be performed in parallel.

Using projection $\Pi$, the saddle point problem (\ref{eq:BDDC-system}) is equivalent to
\begin{equation}
\Pi A^{c}\Pi w_{c}=\Pi R^{c\mathrm{T}}E^{\mathrm{T}}\left( I-P\right) ^{%
\mathrm{T}}r,\quad w_{c}\in\limfunc{null}D^{c}.  \label{eq:BDDC-singular}
\end{equation}
However, the operator $\Pi A^{c}\Pi$ is singular for nontrivial $D^{c}$, so we
solve instead a modified system
\begin{equation}
\left[ \Pi A^{c}\Pi+t(I-\Pi)\right] w_{c}=\Pi R^{c\mathrm{T}}E^{\mathrm{T}%
}\left( I-P\right) ^{\mathrm{T}}r,  \label{eq:BDDC-regular}
\end{equation}
where $t>0$ is a stabilization parameter, e.g. chosen as the maximal diagonal
entry in $A^{c}$. Now, the operator $\Pi A^{c}\Pi+t(I-\Pi)$ is regular,
while the solutions of the systems~(\ref{eq:BDDC-system}) and (\ref%
{eq:BDDC-regular}) are the same.

The projection $\Pi$ enforces constraints that couple all degrees of freedom
on corresponding globs. For this reason, the action of $\Pi$ introduces new
off-diagonal elements (called fill-in) in the projected matrix $\Pi A^{c}\Pi+t(I-\Pi)$. 
This is illustrated in Figure \ref{fig:projections},
where new dense off-diagonal blocks between globs appear. 
Because of these blocks, the performance of sparse direct
solvers would seriously deteriorate.
A~sufficient remedy for this issue is proposed in the next section.

\section{Generalized change of variables}
\label{sec:generalized}

To reduce the fill-in corresponding to the enforcing of the constraints
following (\ref{eq:projection})--(\ref{eq:BDDC-singular}), we revisit and
generalize the change of variables proposed in \cite%
{Klawonn-2006-PID,Li-2006-FBB}. On each substructure $i$, consider first the
change of variables by the transformation
\begin{equation}
w_{i}^{\mathrm{new}}=\bar{H}_{i}w_{i},\quad \bar{H}_{i}=
\left[
\begin{array}{cc}
\bar{U} & \bar{V} \\
0 & I
\end{array}
\right] .  \label{eq:change-variables}
\end{equation}
That is, the averages (given as rows of a matrix $H_{i}^{\mathrm{avg}} = 
\left[
\begin{array}{cc}
\bar{U} & \bar{V}
\end{array}
\right] $)   
are at the beginning of
the vector $w_{i}^{\mathrm{new}}$, replacing the variables in $w_{i}$.
The remaining variables in $w_{i}$ are unchanged. We assume that the vectors
of weights in the averages are linearly independent, that is, $H_{i}^{\mathrm{avg}}$ 
has a full row rank. While this assumption guarantees that
there exists a square submatrix of $H_{i}^{\mathrm{avg}}$ consisting of
linearly independent columns, this does not necessarily need to be the
matrix $\bar{U}$, and so the inverse transformation $\bar{H}_{i}^{-1}$ may not exist.
To correct this, we compute the $QR$ decomposition of $H_{i}^{\mathrm{avg}}$ 
with column pivoting to choose which variables in $w_{i}$ will be replaced by the averages. 
Decompose
\begin{equation}
H_{i}^{\mathrm{avg}}=Q
\begin{bmatrix}
U & V
\end{bmatrix}
K,
\end{equation}
where $Q$ is an orthogonal matrix, $U$ is an upper triangular matrix, 
and $K$ is a~permutation matrix. We now define the generalized change of
variables by
\begin{equation}
w_{i}^{\mathrm{new}}=H_{i}w_{i},\quad H_{i}=\left[
\begin{array}{cc}
U & V \\
0 & I
\end{array}
\right] K.  \label{eq:general-change-variables}
\end{equation}
Now $H_{i}^{-1}$ exists, and the inverse change of variables is defined as
\begin{equation}
T_{i}=H_{i}^{-1}=K^{-1}\left[
\begin{array}{cc}
U^{-1} & -U^{-1}V \\
0 & I%
\end{array}
\right] .  \label{eq:T}
\end{equation}
The matrix $U$, though invertible, is not guaranteed to be well conditioned.
This is a well-known problem in $QR$ decomposition \cite{Langou-2008-PC}.
However, we can drop the rows of $\left[ U\text{ }V\right] $ where the
diagonal entry of $U$ is small, and one can argue that the constraints that
were transformed into rows with negligible leading entry are (numerically)
redundant. Our implementation of the change of variables uses $QR$
decomposition by the LAPACK routine DGEQP3.

To compare with the change of variables from \cite{Klawonn-2006-PID,Li-2006-FBB}, 
consider the case where there is just one
average with unit weights. Then
\begin{equation}
H_{i}^{\mathrm{avg}}=\left[ 1,\dots,1\right] ,\quad U=\left[ 1\right] ,\quad
K=I,\quad H_{i}=\left[
\begin{array}{cccc}
1 & 1 & \ldots & 1 \\
& 1 &  &  \\
&  & \ddots &  \\
&  &  & 1
\end{array}
\right] ,
\end{equation}
and we have the change of variables
\begin{equation}
w_{i}=T_{i}w_{i}^{\mathrm{new}},\quad T_{i}=H_{i}^{-1}=\left[
\begin{array}{cccc}
1 & -1 & \ldots & -1 \\
& 1 &  &  \\
&  & \ddots &  \\
&  &  & 1
\end{array}
\right] ,  \label{eq:TB-example-general}
\end{equation}
while the transformation of variables by \cite{Klawonn-2006-PID,Li-2006-FBB}
is defined as
\begin{equation}
w_{i}=\left[
\begin{array}{cccc}
1 & -1 & \ldots & -1 \\
1 & 1 &  &  \\
1 &  & \ddots &  \\
1 &  &  & 1
\end{array}
\right] w_{i}^{\mathrm{new}}.
\end{equation}

With the change of basis, the BDDC preconditioner can be written as
\begin{equation*}
M_{BDDC}:r\longmapsto u=\left( I-P\right) ETR\overline{w},\quad\overline {w}%
\in\widetilde{W}:\quad\left\langle AT\overline{w},Tz\right\rangle
=\left\langle r,\left( I-P\right) ETz\right\rangle ,\quad\forall z\in%
\widetilde{W},
\end{equation*}
where $T=\limfunc{diag}\left[ T_{i}\right] $. Thus, $A$ is replaced by the
transformed matrix $T^{\mathrm{T}}AT$, and, by assembly at corners following
(\ref{eq:c-subassembly}), $A^{c}$ becomes $R^{c\mathrm{T}}T^{\mathrm{T}%
}ATR^{c}$. Then, Lemma \ref{lem:saddle} yields the matrix form of the
algorithm: solving the system
\begin{equation}
\begin{array}{ccccc}
R^{c\mathrm{T}}T^{\mathrm{T}}ATR^{c}\overline{w}_{c} & + & \overline{D}%
^{cT}\lambda & = & R^{c\mathrm{T}}T^{\mathrm{T}}E^{\mathrm{T}}r, \\
\overline{D}^{c}\overline{w}_{c} &  &  & = & 0,%
\end{array}
\label{eq:BDDC-ch-var}
\end{equation}
followed by computation of the approximate solution $u\in\widehat{W}$ by 
$u=\left( I-P\right) ETR^{c}\overline{w}_{c}$. Here, the matrix 
$\overline{D}^{c}=D^{c}T$ 
is much sparser than $D^{c}$ thanks to the change of
variables. 
It couples only the new explicit degrees of freedom on each subdomain
and thus has only one $+1$ and one $-1$ entry on each row. In fact, the
construction of $\overline{D}^{c}$ is similar to the construction of the
operator $B$ used in FETI methods. In computations, $\overline{D}^{c}$ can
be constructed directly without using either $D^{c}$ or $T$, knowing only
which pairs of the (explicit) interface degrees of freedom should be coupled after the
change of basis.

Instead of solving the saddle point problem (\ref{eq:BDDC-ch-var}) directly,
we now use the projection as in~(\ref{eq:BDDC-regular}) with $D^{c}$
replaced by $\overline{D}^{c}$, resulting in a~new projection~$\overline{\Pi}$.
The sparsity structure of $\overline{\Pi}$ and of the projected
matrix $\overline{\Pi}R^{c\mathrm{T}}T^{\mathrm{T}}ATR^{c}\overline{\Pi}+%
\overline {t}(I-\overline{\Pi})$ are illustrated Figure \ref{fig:projections}%
. As can be observed, change of basis preceding the projection can lead to
much lower fill-in in the off-diagonal blocks of the projected matrix.

The BDDC preconditioner can be finally rewritten in the algebraic form, 
which is actually used in our implementations, as follows.

\begin{algorithm}
The action of the BDDC preconditioner $M_{BDDC}:r\mapsto w$ with the
generalized change of variables consists of solving the system
\begin{equation}
\widetilde{A}\overline{w}_{c}=\overline{\Pi}R^{c\mathrm{T}}T^{\mathrm{T}}E^{\mathrm{T}}r,
\end{equation}
where
\begin{equation}
\widetilde{A}=\overline{\Pi}R^{c\mathrm{T}}T^{\mathrm{T}}ATR^{c}\overline{%
\Pi }+t(I-\overline{\Pi}),  \label{eq:A-tilde}
\end{equation}
with an arbitrary $t>0$, followed by $w=\left( I-P\right) ETR^{c}\overline{w}%
_{c}$.
\end{algorithm}

\begin{remark}
Since the transformation of variables changes averages into separate degrees
of freedom, one can treat these degrees of freedom as corners and assemble
them just as in~\cite{Klawonn-2006-PID,Klawonn-2006-DPF,Li-2006-FBB} to make
all constraints primal. This gives no additional fill-in beyond the one
caused by the change of variables, i.e., replacing $A^{c}$ by $R^{c\mathrm{T}%
}T^{\mathrm{T}}ATR^{c}$. In the adaptive method (Section \ref%
{sec:adaptive-constraints}), the corners are already set and used to compute
the constraints to be added adaptively. Treating all constraints as corners
then requires redefining which variables are corners. This is not supported
in the code described here. See Section \ref{sec:implement} for more details.
\end{remark}

\section{Implementation}

\label{sec:implement}

We have implemented the proposed method in \textsc{Matlab}. Later, we
have developed also a~parallel version using Fortran 90 programming language
and MPI.

First, we have implemented the BDDC preconditioner based on the formulation (%
\ref{eq:BDDC-system}). In the case of corner constraints only, i.e. $D^{c}$ is
empty, the method is reduced to solving a~problem with matrix $A^{c}$ in
each iteration. In the current version, we rely on the parallel direct
solver MUMPS \cite{Amestoy-2000-MPD} (version 4.8.4) for this purpose.

In the implementation, two separate instances of MUMPS are necessary -- one
for solving problems with matrix $A^{c}$ and another for a realization of
the operator $I-P$ of the discrete harmonic extension in (\ref{eq:harm-E})
globally. 
The latter is equivalent to solving an independent discrete Dirichlet problem on each
subdomain.

In the case of a nontrivial matrix $D^{c}$, i.e. for additional constraints on
edges and/or faces, explicit change of variables with projection (Section \ref{sec:generalized}) 
is performed in parallel to form the distributed
sparse matrix (\ref{eq:A-tilde}), which is then supplied to MUMPS. We have
observed a~great advantage in projecting the matrix after the change of
variables compared to the direct projection on $\mathrm{null}\ D^{c}$. It
significantly decreases the computational time and memory consumption due to 
the reduced fill-in, as
described in Section \ref{sec:generalized}. In our experience, the 
amount of extra work needed for the transformation and the projection is typically
only a~small fraction of the time saved by the lower number of PCG iterations,
compared to the case of corner constraints only.

Using the projection instead of re-assembling the matrix after the
change of variables allows us to store the sparse matrix in memory only
once, and use it in the preconditioner as well as in the PCG method, which
is formulated to run on vectors from the space $\widetilde{W}^{c}$.
For the preconditioner, new entries
arising from the transformation and projection are stored in the memory
behind the original matrix and the convention of repeated indices allowed by
MUMPS is exploited.

Later, the adaptive selection of constraints described in Section \ref{sec:adaptive-constraints} 
has been added to the implementation. As the
parallelization of solving the generalized eigenvalue problems (\ref{eq:eig-matrix}) 
on pairs of adjacent subdomains does not follow the scheme
of the natural parallelization by subdomains, this part of the code has been
written as a self-standing module that just passes the constraints to the
main BDDC solver. Multiplication by $S^{c}_{ij}$ in the eigenvalue problem (\ref{eq:eig-matrix}) 
is implemented by performing the interior correction on each of
the two adjacent subdomains separately, and only the resulting vectors are assembled;
thus, the matrices $S^{c}_{ij}$ and $A^{c}_{ij}$ for the two adjacent substructures
are not formed explicitly.

\section{Numerical results}
\label{sec:numerical}

We have tested the adaptive algorithm on several
three-dimensional problems of linear elasticity coming from engineering
practice. As a consistency check, we have also tested the method in two
dimensions with essentially the same results as in~\cite{Mandel-2007-ASF}.
The computations were done in \textsc{Matlab} and by the parallel
implementation described in Section~\ref{sec:implement}. 
In \textsc{Matlab}, the generalized
eigenvalue problems for pairs of adjacent substructures were solved by
explicit construction of the matrices and by standard methods for 
symmetric eigenvalue problems. 
We have also tested both the \textsc{Matlab} and the \textsc{C} versions 
of the LOBPCG algorithm~\cite{Knyazev-2001-TOP}. The
averaging operator was constructed with weights proportional to the diagonal
entries of the substructure matrices before elimination of interiors.

The first problem is a nozzle box of a \v{S}KODA steam turbine $28$ MW for
the electric power plant Nov\'{a}ky, Slovakia, loaded by steam pressure. The
body of the nozzle box was discretized using $2696$ isoparametric
quadratic finite elements with $40,254$ degrees of freedom and decomposed
into $16$ substructures with $37$ corners, $19$ edges, and $32$ faces (see Fig.~\ref{fig:jettube}). 
Convergence of the algorithm with non-adaptive
constraints is displayed in Table~\ref{tab:jettube-nonadaptive}. Note that
the corner coarse degrees of freedom were not sufficient to guarantee convergence. 
Where ``3eigv'' is added, constraints corresponding to three dominant eigenvalues 
are added at each face.
This choice leads to the same number of constraints as using simple arithmetic averages.
Comparing the last two rows in Table~\ref{tab:jettube-nonadaptive},
we can see that constraints obtained from the adaptive
algorithm work quite better than arithmetic averages. Our explanation is
that such constraints might approximate better the direction of global
eigenvectors corresponding to the extreme eigenvalues. Table~\ref{tab:jettube-adaptive} 
then contains results obtained using the adaptive
selection of constraints. Each row corresponds to a different value of the
threshold $\tau$, i.e. the target bound of the condition number indicator $\widetilde{\omega }$. 
All eigenvectors corresponding to eigenvalues greater or
equal to $\tau$ were used to generate adaptive constraints. Comparing the
results in Tables~\ref{tab:jettube-nonadaptive}--\ref{tab:jettube-adaptive},
we can see that the adaptive method leads to a redistribution of the number of
constraints on different faces. 
For example, with $\tau=20$, the total
number of constraints is still lower than by using arithmetic averages on all faces, 
while the number of iterations is improved by almost $25\%$ and the condition number estimate$%
~\kappa$ is improved by more than $50\%$.

The second problem is a beam with a mesh refinement around a notch.
It is discretized using $245,687$ tetrahedral finite elements with $143,451$
degrees of freedom, and decomposed into $8$~substructures with $31$ corners,
$18$ edges, and $19$ faces (see Fig.~\ref{fig:beam}). The results with
non-adaptive constraints are summarized in Table~\ref{tab:beam-nonadaptive},
and results of the adaptive method are presented in Table~\ref{tab:beam-adaptive}.
Comparing these two tables, we can see that, similarly as for the nozzle box
problem, doubling the number of constraints reduces number of iterations to
a half. Nevertheless, for both problems, the adaptive algorithm leads to a
relatively small improvement in terms of the number of iterations and of the condition
number estimate. This indicates that for these problems the simple
arithmetic averages already work well enough, and there are no interfaces that
would require extra work - the quality of the decomposition is uniform, as
seen in Figs.~\ref{fig:jettube}--\ref{fig:beam}.

On the other hand, the power of the adaptive algorithm seems to be very beneficial 
for finite element discretization with bad aspect ratios of elements. 
An example of such problem is a
bridge construction discretized by $39,060$ hexahedral finite elements with
$157,356$ degrees of freedom, and distributed into $16$ substructures with 
$250$ corners, $30$ edges, and $43$ faces (see Fig.~\ref{fig:bridge}). The
results are summarized in Tables~\ref{tab:bridge-nonadaptive}--\ref{tab:bridge-adaptive}. 
Comparing the last two rows in Table~\ref{tab:bridge-nonadaptive} 
we can see, that relatively poor convergence with
arithmetic averages improves quite significantly when arithmetic averages
over faces are replaced by the same number of adaptive averages. Moreover,
from Table~\ref{tab:bridge-adaptive} we see that, e.g., doubling the number
of constraints with $\tau=5$ decreases the number of iterations more than
six times, and with $\tau=2$ the number of iterations is reduced more than
ten times while the number of constraints increases approximately four times.

In order to test the performance of the algorithm in the presence of jumps
in material coefficients, we have created a problem of a cube with material
parameters $E=10^6$ Pa and $\nu=0.45$, penetrated by four bars with
parameters $E=2.1\cdot10^{11}$ Pa and $\nu = 0.3$, consisting of $107,811$
degrees of freedom, and distributed into $8$ substructures with $30$
corners, $16$ edges, and $15$ faces (see Fig.~\ref{fig:compo8} and note that the
bars cut the substructures only through faces). Similar problems are solved
in practice to determine numerically (anisotropic) properties of composite
materials. Comparing the results in Tables~\ref{tab:compo8-nonadaptive} and~%
\ref{tab:compo8-adaptive} we can see, that with $\tau=10,000$ only $10$
additional averages over faces are used to decrease the number of iterations
$2.6$ times. 
With $\tau=2$, the number of iterations is decreased $10$ times
compared to the non-adaptive algorithm with arithmetic averages over all globs (c+e+f), 
whereas the number of constraints is increased less than $3$ times.

To test the parallel behaviour of the MUMPS solver at our applications, we
run the solver only with corner coarse degrees of freedom on a benchmark
problem consisting of cubic subdomains, the number of which is growing in two
dimensions (see Fig.~\ref{fig:planar_cubes}). Since the size of the
subdomains is fixed, the problem fixed at one side and loaded at the opposite
side is changing its nature, and consequently, some growth in number of iterations is
expected. The sequence of problems is run on an increasing number of processors,
which matches the number of subdomains.
Presented computations were performed on 1.5~GHz Intel Itanium~2
processors of SGI Altix~4700 computer in CTU Supercomputing Centre, Prague.

We can see in Table~\ref{tab:scaling_cornersHh8}, that after a jump in times
between 16 and 25 subdomains related probably to the computer's
architecture, the times of analysis, factorization, as well as per one
iteration remain almost constant. The ``total wall time'' includes also the
second factorization of MUMPS for computing the discrete harmonic extensions
and all I/O operations.

In the presented algorithm, a considerable amount of work is spent by
generating the adaptive constraints. This part, which consists of solving local
eigenproblems, may eventually dominate the whole computation. Each
of these eigenproblems is common to a~pair of subdomains, so the 
parallelism is different from the ``natural'' subdomain-wise parallelism of domain decomposition.
For this reason, an independent distribution of eigenproblems among processors
is performed, and each processor which solves an eigenproblem is linked with
the two processors which store the data of subdomains within the pair.

An investigation of scaling of this part of the implementation on a~variable number of processors
was performed for the problem of turbine nozzle box with 16 subdomains. 
In Fig.~\ref{fig:scaling_eigenproblems_select},
the summary for the sequence of $2^{k},\ k = 0,1,2,\dots,5$, processors is shown.

\section{Conclusion}
\label{sec:conclusion}

The adaptive BDDC method has been presented. The paper contains several
original contributions. First, the definition of space where BDDC runs is
given as the nullspace of a~global matrix of constraints. For an efficient and
straightforward implementation of this formulation, a generalization of
the change of variables is proposed. This allows an efficient handling of multiple
arbitrary constraints on a substructure face. This functionality is required
for the implementation of constraints that are generated adaptively. The
adaptive selection of constraints from \cite{Mandel-2007-ASF} has been
reformulated in a mathematically equivalent way to use only the operators of
BDDC and to match the overall approach of the rest of this paper to minimize
programming requirements. The adaptive method is based on simultaneous
solution of generalized eigenvalue problems defined for each face in the
decomposition. The eigenvalues serve as a condition number indicator, so the
minimal number of constraints is added to guarantee that the condition
number indicator is below a given threshold. The corresponding eigenvectors
are used to derive the coefficients of the constraints. Numerical
experiments confirm that the eigenvalues provide a good prediction of the
final condition number of the preconditioned operator.

A parallel implementation of the method has been developed and presented. It
is based on a global formulation of the matrix of the BDDC preconditioner,
and it is built on top of solver MUMPS, which provides most of the
parallelism and minimizes custom coding. The implementation has been tested
on a number of problems of 3D elasticity. Results for several real world
problems are included.

In our experiments, adaptive BDDC has shown to be quite powerful. Many
times, it has been able to save the situation for poorly selected
corners, even in the case of disconnected subdomains, the situation often
faced in real applications of domain decomposition when using graph
partitioners. Adaptive BDDC is able to handle very ill-conditioned problems
(e.g. problems with jumps in coefficients, complicated geometries with
deformed elements, etc.), which are almost impossible to solve by standard
BDDC method using only arithmetic averages on edges and faces. Such problems
would either require a~prohibitive number of PCG iterations or may not
converge at all. This class of problems is the target application of the
proposed method, as the extra cost of generating the constraints adaptively
is not negligible and would not pay for well-conditioned problems.

The solution of local eigenproblems by LOBPCG in generation of the adaptive
constraints requires many iterations and accounts for most of the time 
for some problems. A~suitable preconditioning of these eigenproblems to reduce the number
of LOBPCG iterations will be studied elsewhere.

\section*{Acknowledgement}

We would like to thank to Jaroslav Kruis, Jaroslav Novotn\'{y}, and Jan Le%
\v{s}tina for providing us with real engineering problems. We are also
grateful to Marian Brezina for visualization of some meshes by his own
package \emph{MyVis}. Part of this work was done when Jakub \v{S}\'{\i}stek was
visiting University of Colorado Denver.

This work was supported in part by National Science Foundation under grants 
\mbox{DMS-0713876} and \mbox{CNS-0719641}, by Czech Science Foundation
under grant GA \v{C}R 106/08/0403, and by 
Academy of Sciences of the Czech Republic under grant \mbox{AV0Z10190503}.

{\raggedright
\bibliographystyle{elsart-num-sort}
\bibliography{../../bibliography/bddc,adaptive}
}


\newpage

\begin{figure}[tbp]
\begin{center}
\includegraphics[width=70mm]{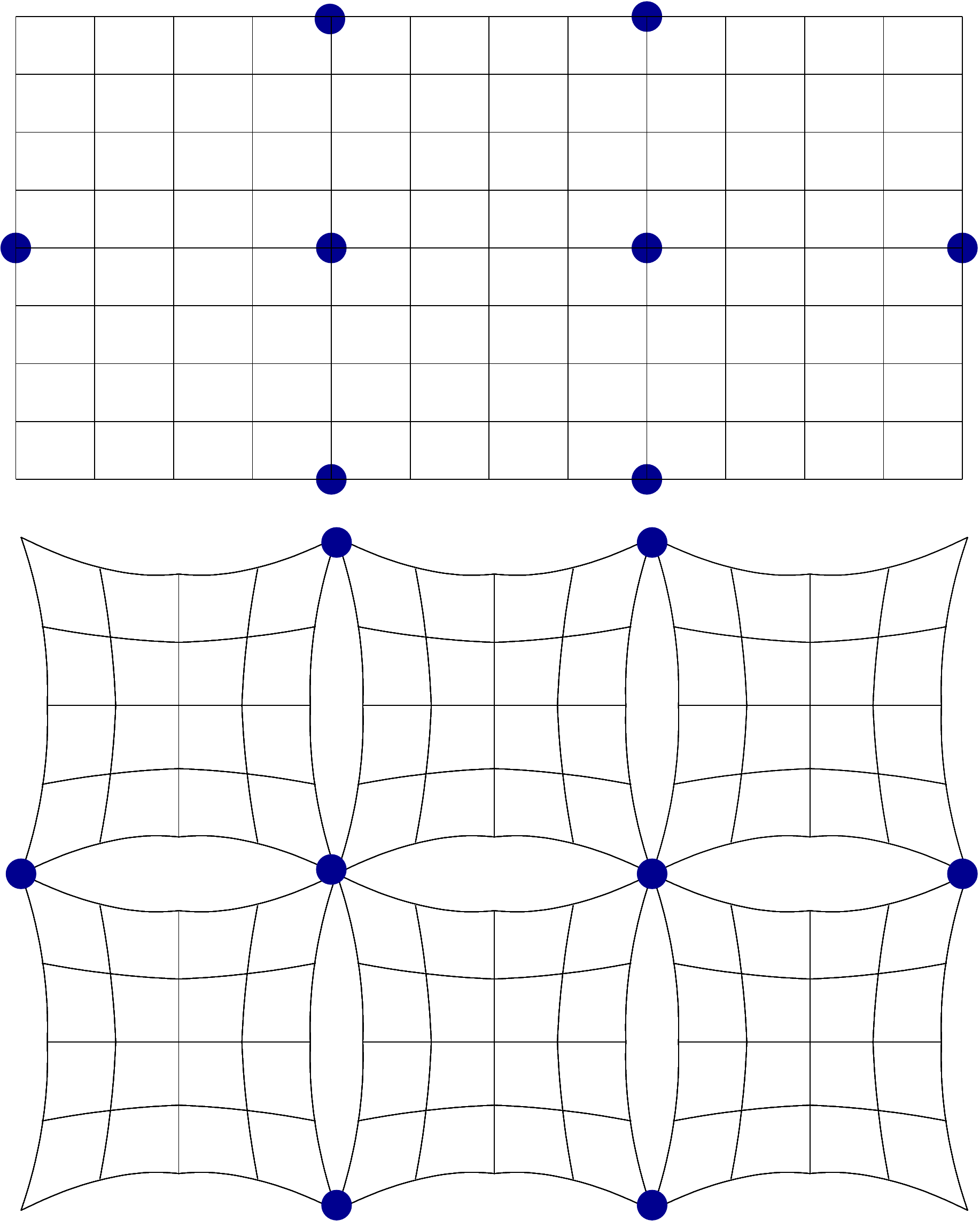}
\end{center}
\caption{Example of an actual mesh (top) and the corresponding fictitious
mesh for construction of space $W^{c}$ (bottom), the dots mark corners.}
\label{fig:virtual_mesh}
\end{figure}

\begin{figure}[tbp]
\begin{center}
\includegraphics[width=80mm]{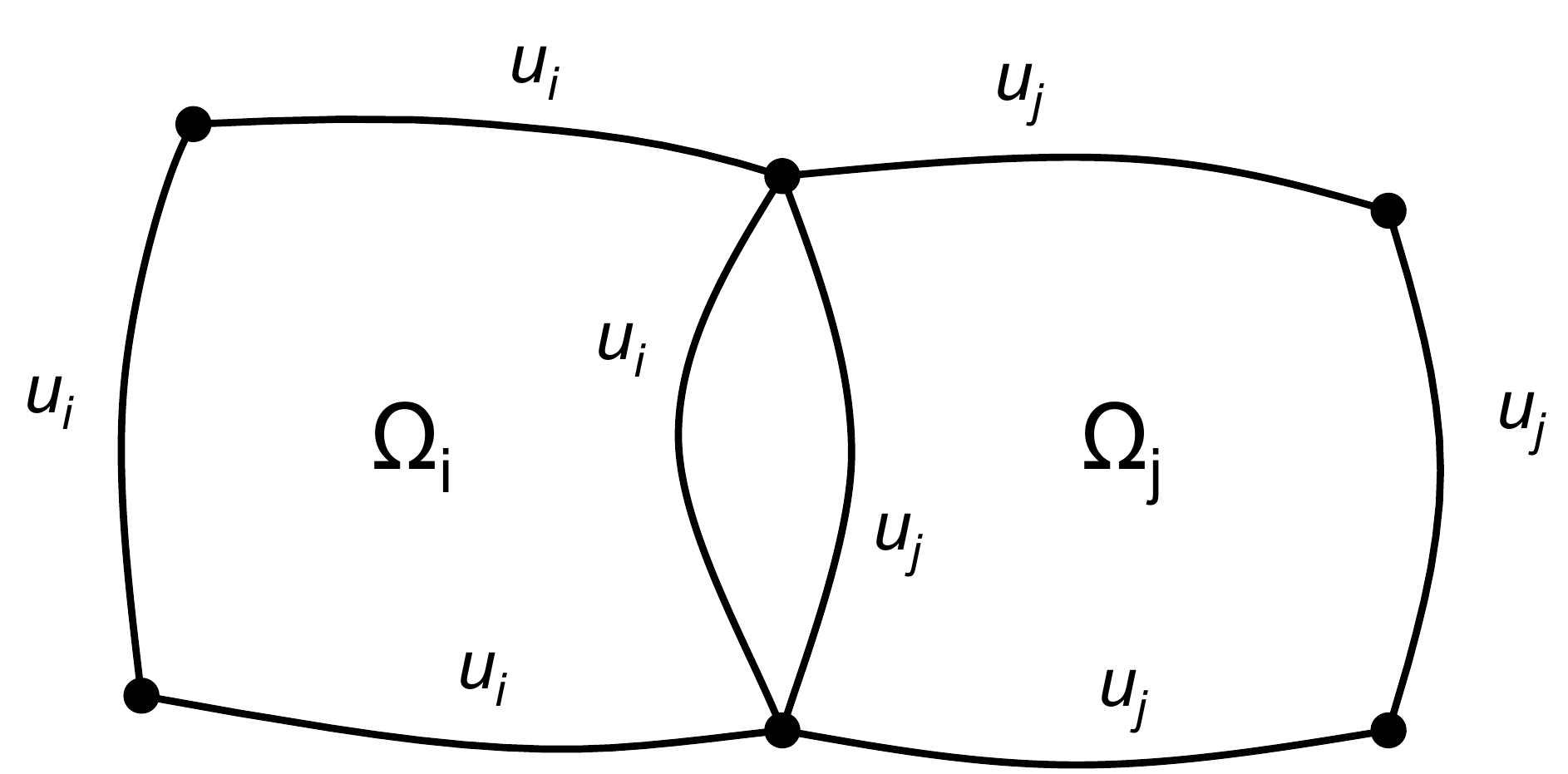}
\end{center}
\caption{Illustration of two adjacent subdomains in 2D for the computation
of the condition number indicator. }
\label{fig:2domains}
\end{figure}

\begin{figure}[tbp]
\begin{center}
\begin{tabular}{cc}
\includegraphics[width=2in]{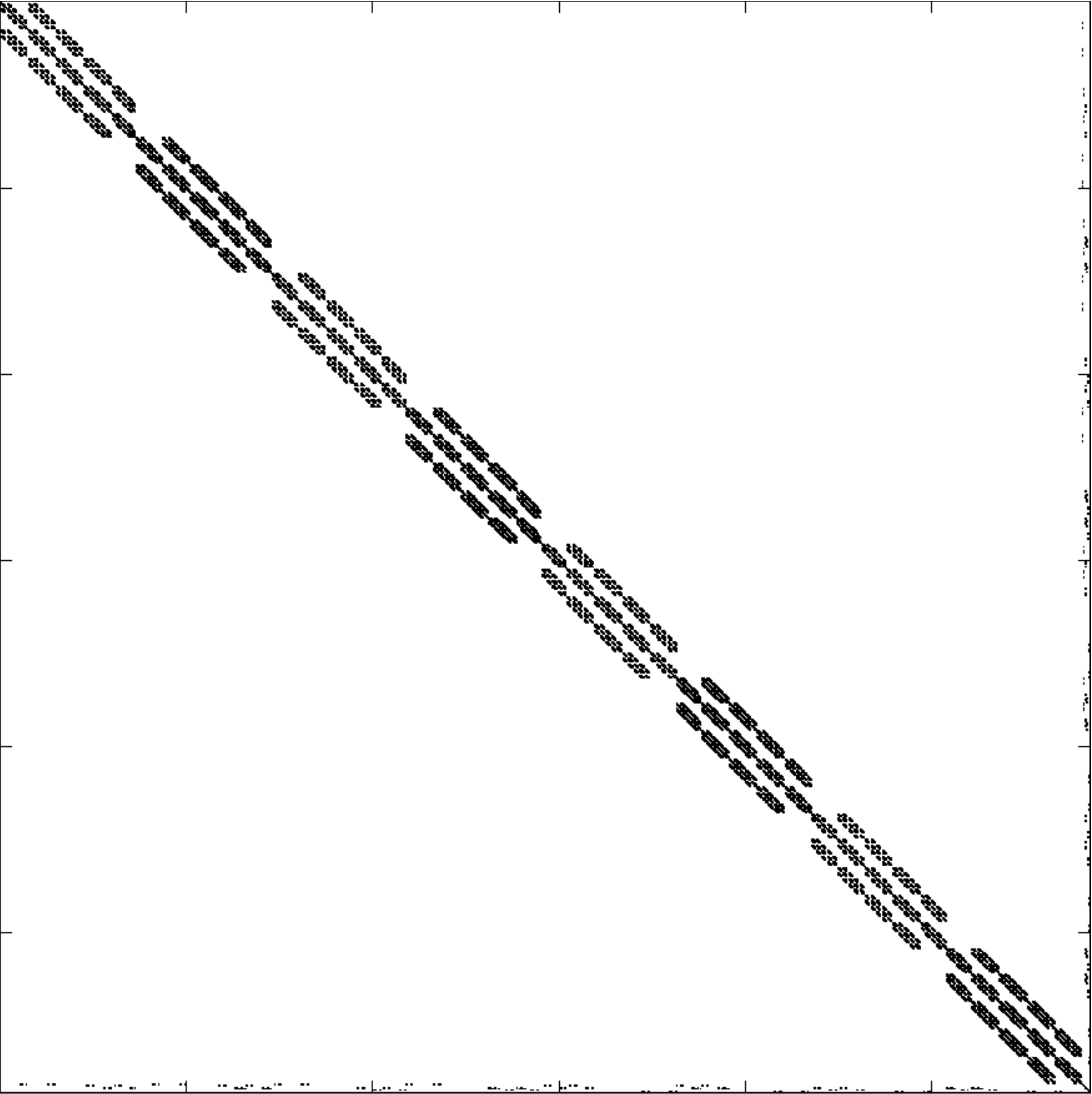} & %
\includegraphics[width=2in]{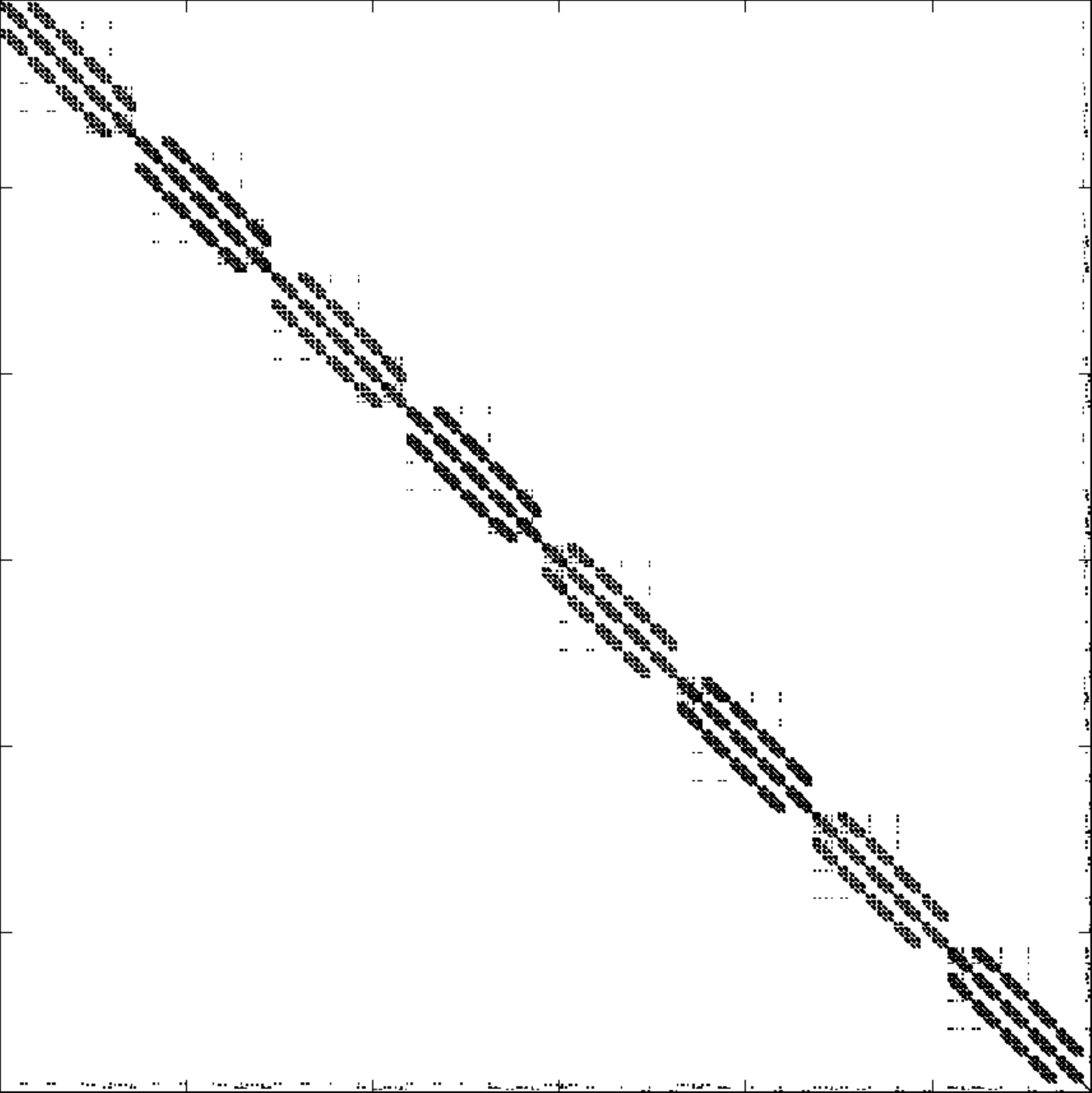} \\
{\small (a) nnz = $136,937$} & {\small (b) nnz = $141,773$} \\
&  \\
\includegraphics[width=2in]{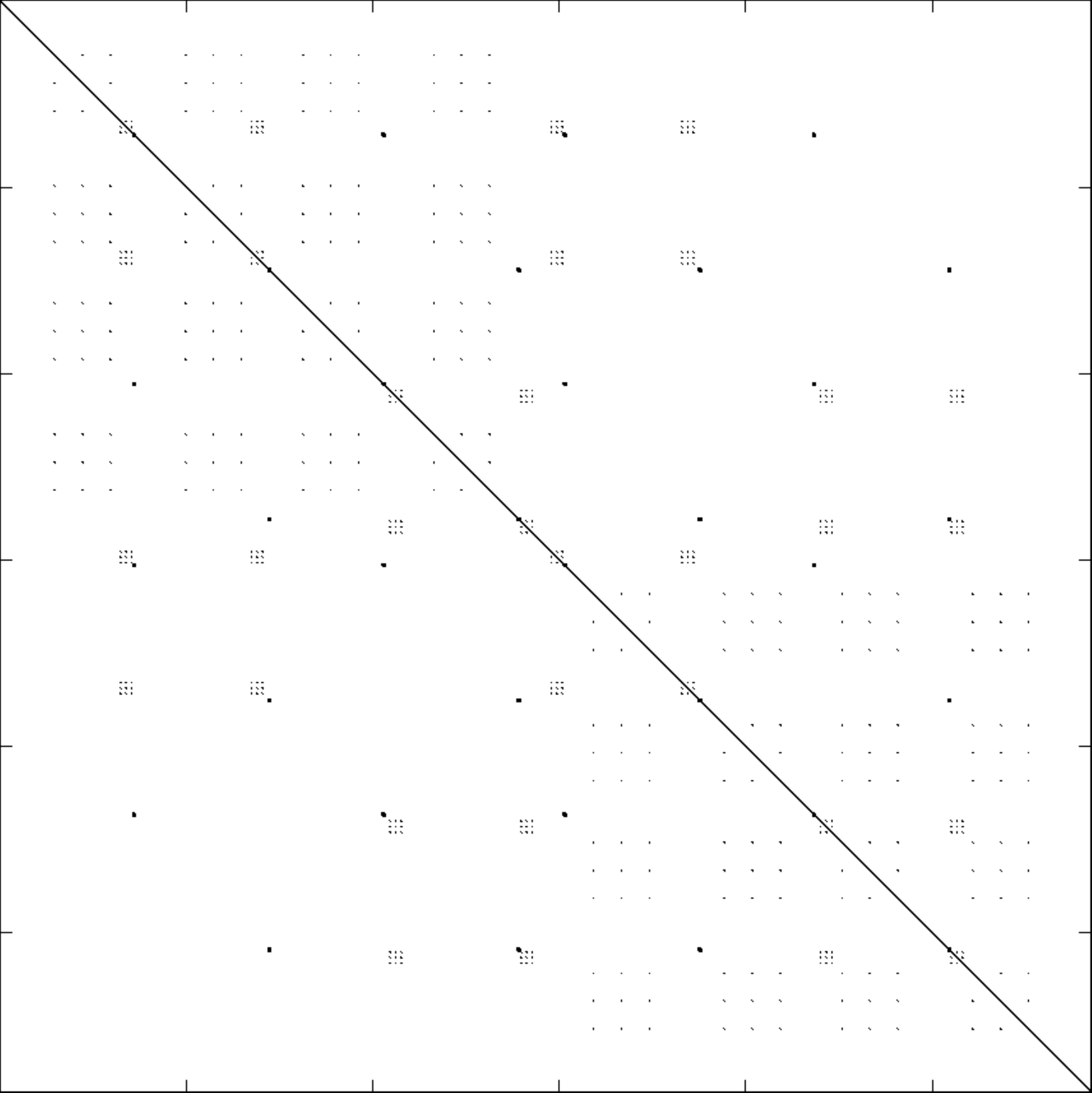} & %
\includegraphics[width=2in]{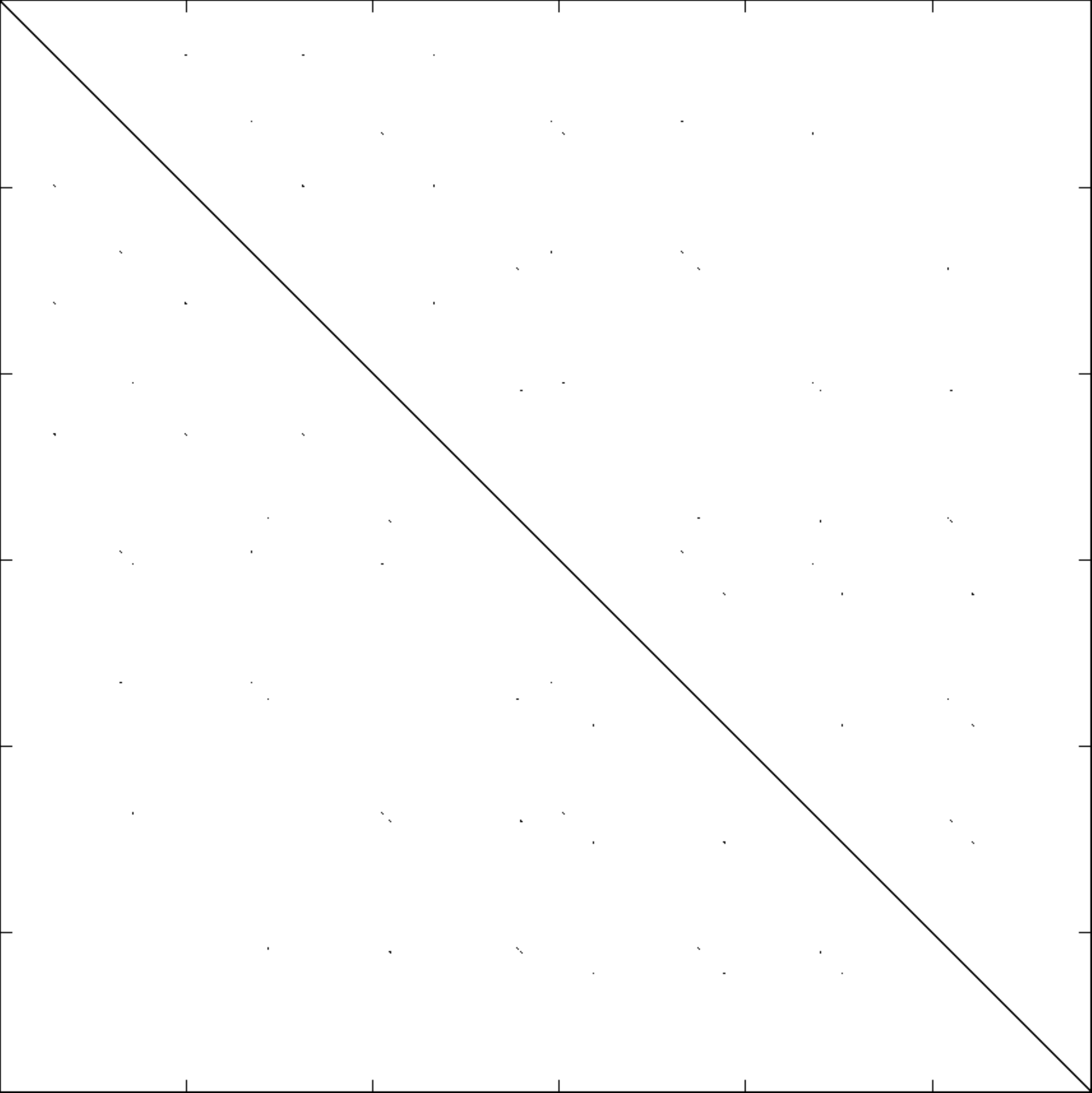} \\
{\small (c) nnz = $5301$} & {\small (d) nnz = $3141$} \\
&  \\
\includegraphics[width=2in]{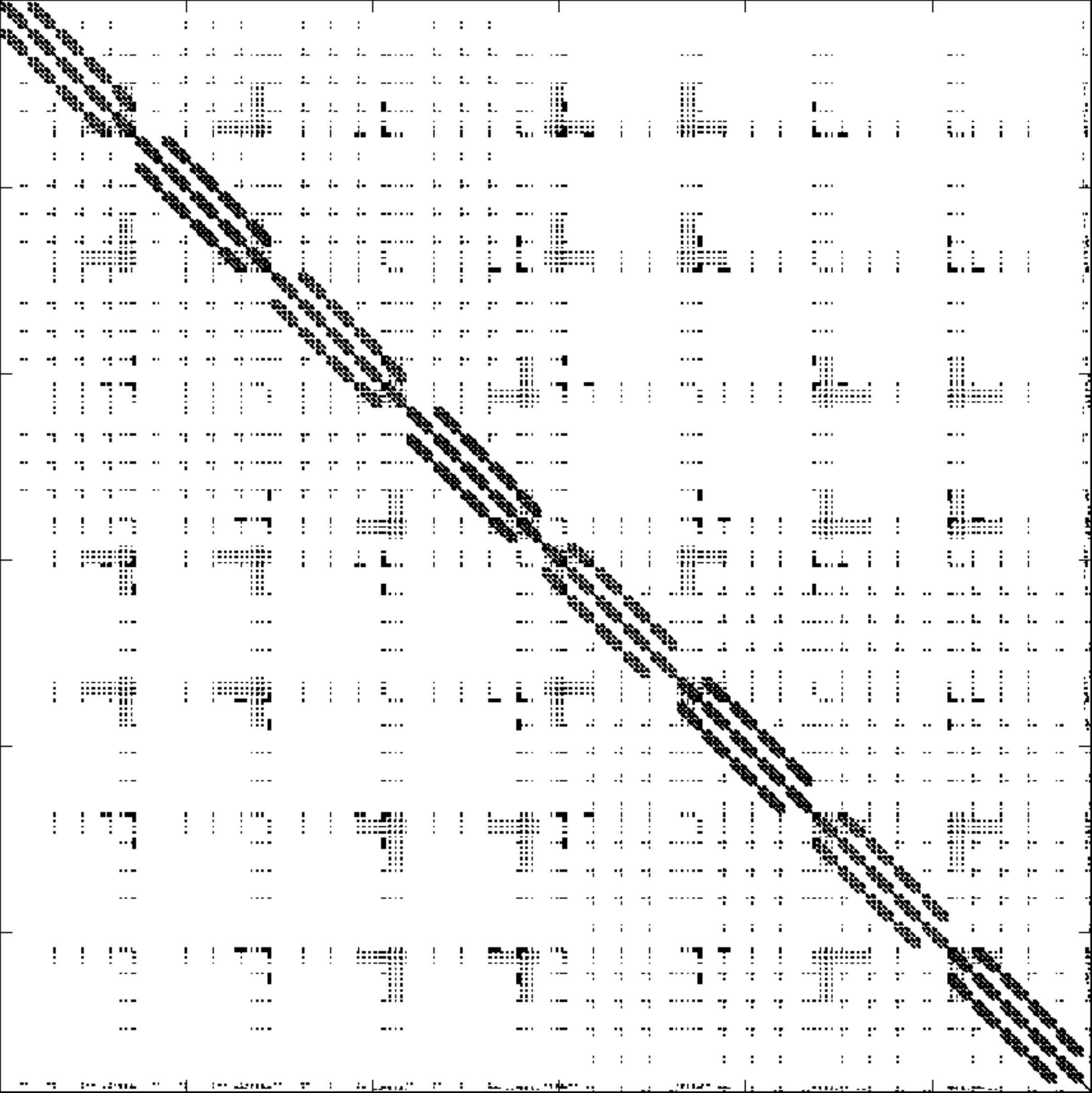} & %
\includegraphics[width=2in]{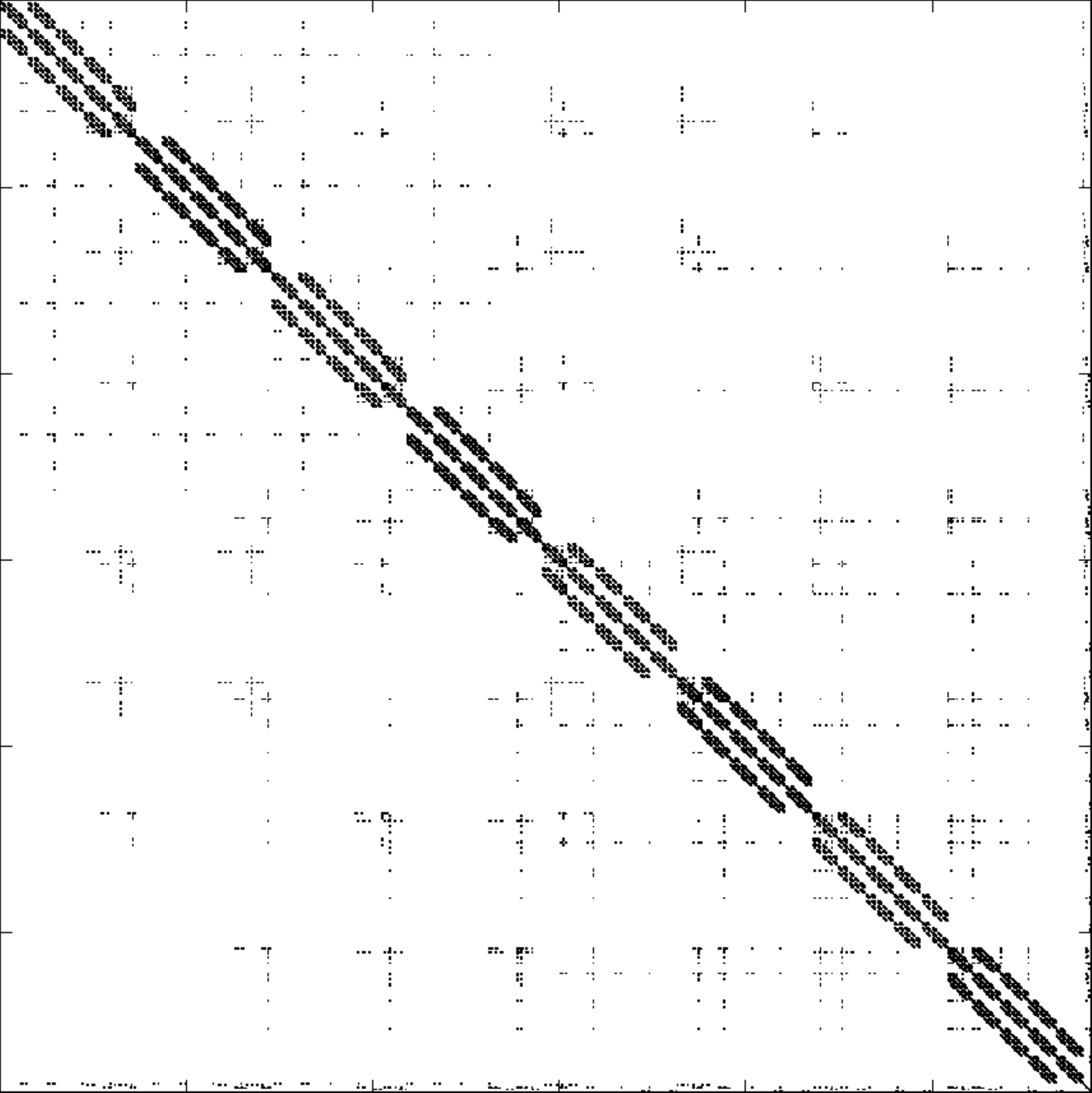} \\
{\small (e) nnz = $228,954$} & {\small (f) nnz = $156,982$}%
\end{tabular}%
\end{center}
\caption{Sparsity patters for 3D elasticity problem for a cube decomposed
into $2\times2\times2$ substructures ($H/h=4$) with $7$ corners, $6$ edges, $%
12$ faces, and $2187$ degrees of freedom. The matrix $D^{c}$ (resp. $%
\overline{D}^{c}$) contains $54$ rows to enforce the equality of arithmetic
averages over edges. The matrices (a), (c), (e) are in the original degrees
of freedom, while (b), (d), (f) are after the change of variables (\ref{eq:T}): 
the operators ${A}^{c}$ in panel (a) and $R^{c\mathrm{T}}T^{%
\mathrm{T}}ATR^{c}$ in panel (b), projections $\Pi$ in panel (c) and $%
\overline{\Pi}$ in panel (d), and projected operators $\Pi{A}%
^{c}\Pi+t(I-\Pi) $ in panel (e) and $\overline{\Pi}R^{c\mathrm{T}}T^{\mathrm{%
T}}ATR^{c}\overline{\Pi}+\overline {t}(I-\overline{\Pi})$ in panel (f). All
are square matrices with size $2925$.}
\label{fig:projections}
\end{figure}

\clearpage 

\begin{figure}[tbp]
\begin{center}
\begin{tabular}{cc}
\includegraphics[width=60mm]{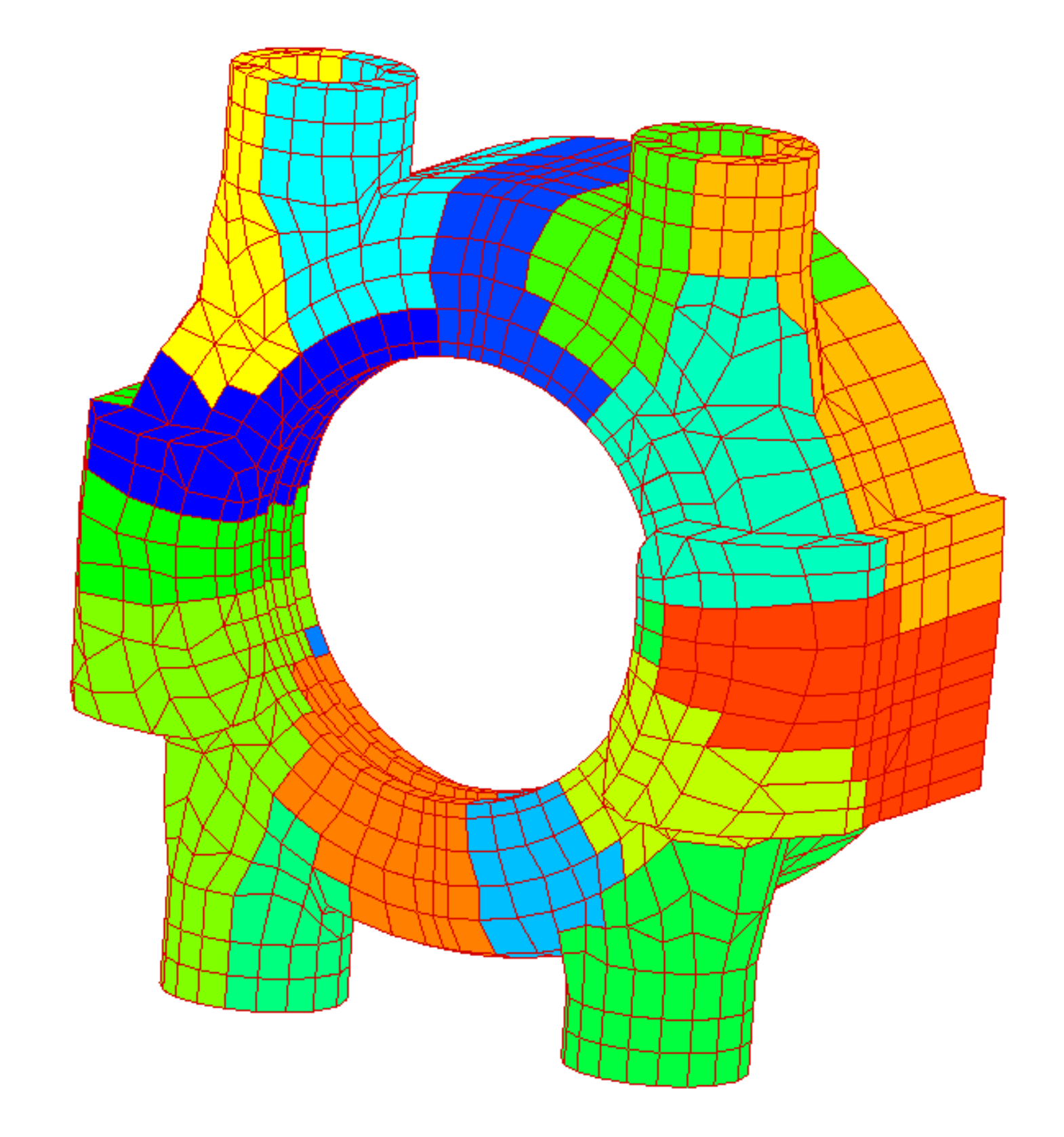} & %
\includegraphics[width=60mm]{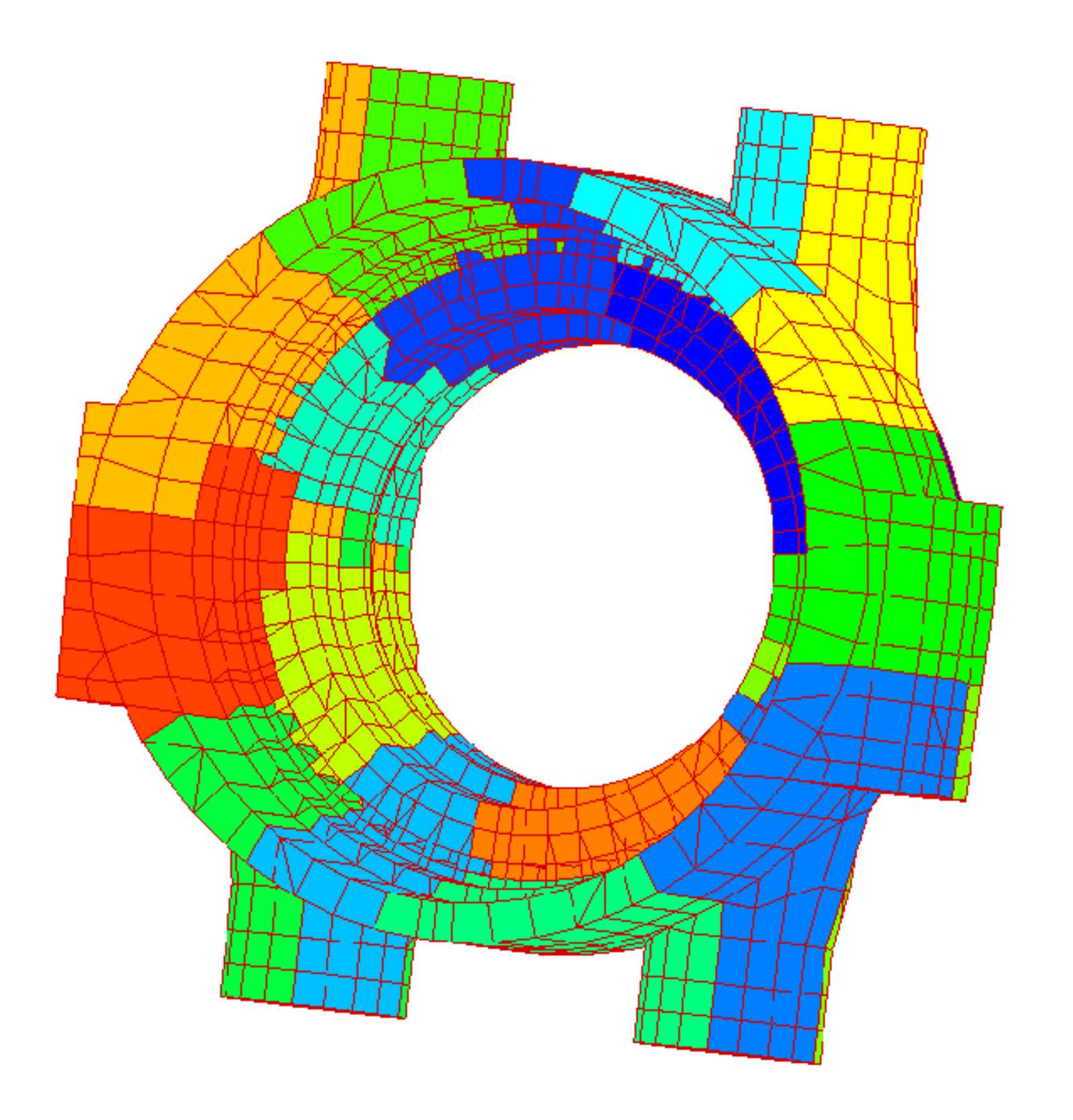} \\
\includegraphics[width=60mm]{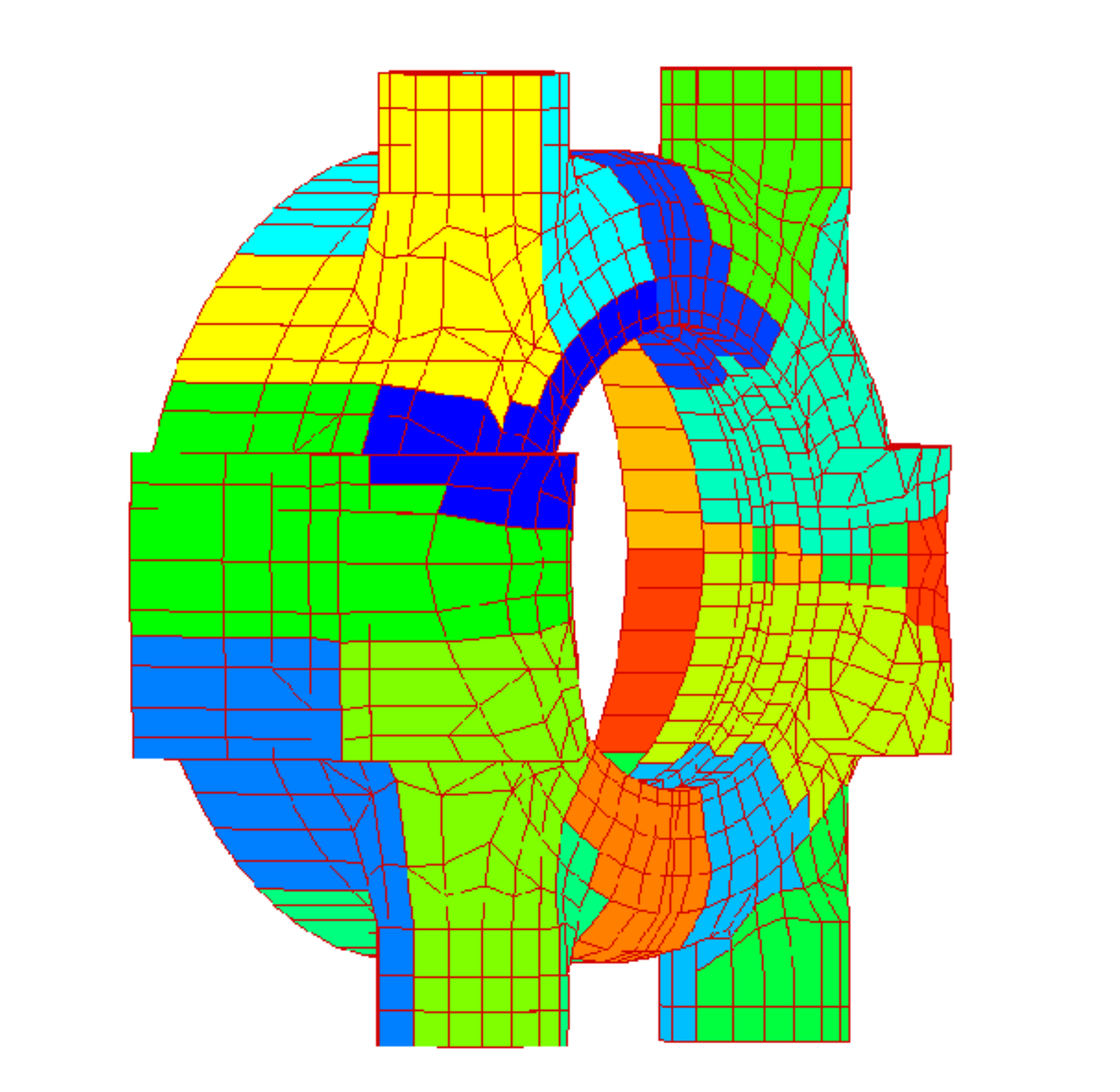} & %
\includegraphics[width=55mm]{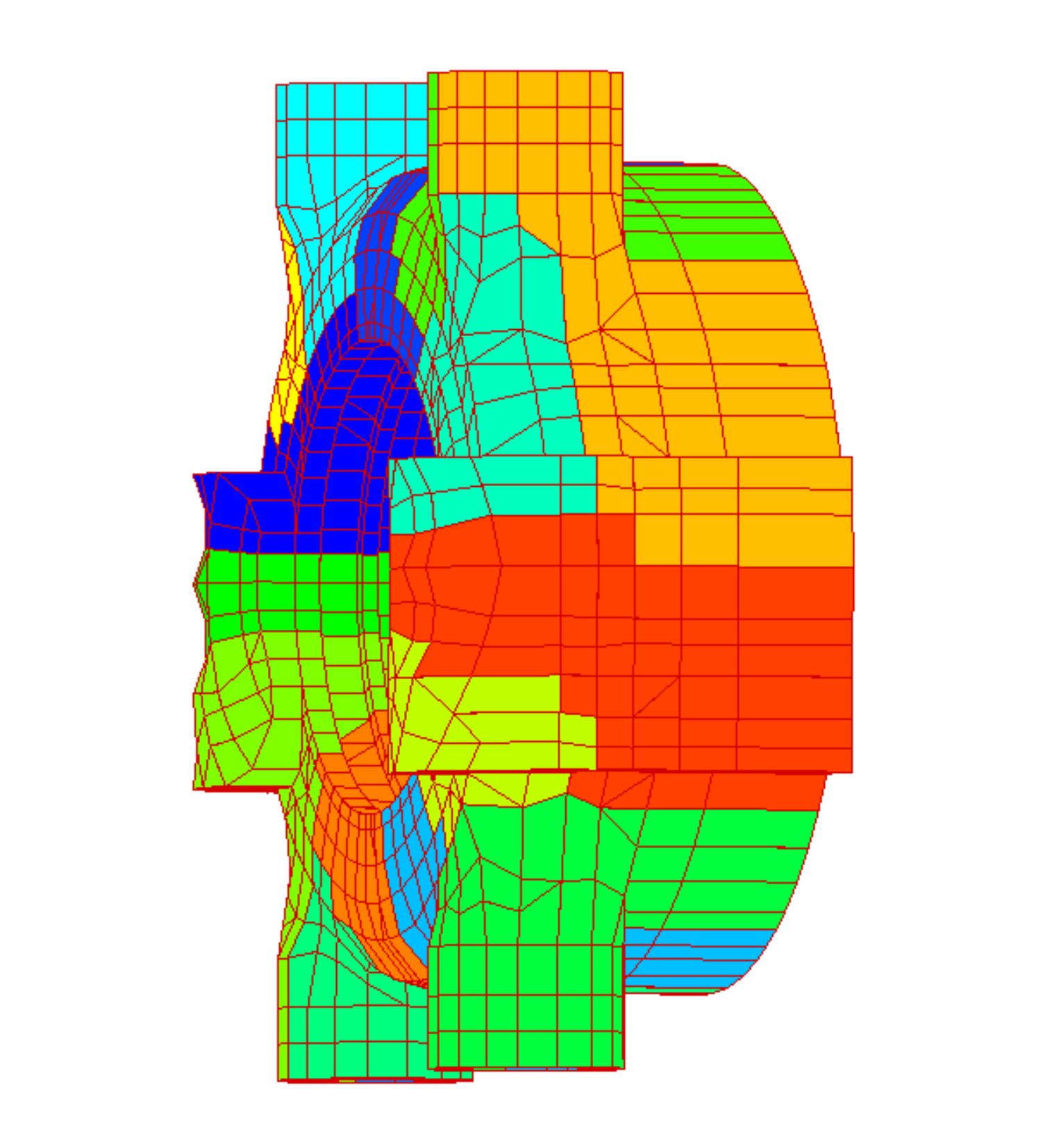}%
\end{tabular}%
\end{center}
\caption{Finite element discretization and substructuring of the nozzle box,
consisting of $40,254$ degrees of freedom, $16$ substructures, $37$
corners, $19$ edges, and $32$ faces.}
\label{fig:jettube}
\end{figure}

\begin{table}[tbp]
\centering
\begin{tabular}{c|r|r|r}
constraints & $Nc$ & $\kappa$ & $it$ \\ \hline
c & $0$ & NA & NA \\
c+e & $117$ & $1021.7$ & $103$ \\
c+e+f & $213$ & $40.3$ & $47$ \\
c+e+f (3eigv) & $213$ & $26.5$ & $40$%
\end{tabular}%
\caption{Results for the turbine nozzle box problem. The first three rows
correspond to non-adaptive approach with corner constraints and arithmetic
averages over edges/faces, and the last row corresponds to corner
constraints with arithmetic averages over edges and three weighted averages
over faces obtained from eigenvectors of the local generalized eigenvalue
problems, $Nc$ is number of constraints (rows in the matrix~$D$), $\protect%
\kappa$ is the approximate condition number estimate from the Lanczos sequence
in conjugate gradients, and $it$ is the number of iterations for relative
residual tolerance~$10^{-8}$.}
\label{tab:jettube-nonadaptive}
\end{table}

\begin{table}[ptb]
\centering
\begin{tabular}{r|r|r|r|r}
$\tau$ & $\widetilde{\omega}$ & $Nc$ & $\kappa$ & $it$ \\ \hline
$\infty${\small {(=c+e)}} & NA & $117$ & $1021.690$ & $103$ \\
$50$ & $49.772$  &  $158$ & $44.8781$ &  $48$ \\
$20$ &  $19.824$ & $200$ & $16.8938$ & $36$ \\
$10$ & $9.965$ & $274$  & $11.171$ & $27$ \\
 $5$  & $4.998$ & $408$ & $8.820$ & $20$
\end{tabular}%
\caption{Results for the turbine nozzle box problem using the adaptive
approach. $\protect\tau$ is the threshold, and $\widetilde{\protect\omega}$
is the condition number indicator from (\protect\ref{eq:cond-ind}). The
other headings are same as in Table~\protect\ref{tab:jettube-nonadaptive}.}
\label{tab:jettube-adaptive}
\end{table}

\begin{figure}[tbp]
\begin{center}
\includegraphics[width=9cm]{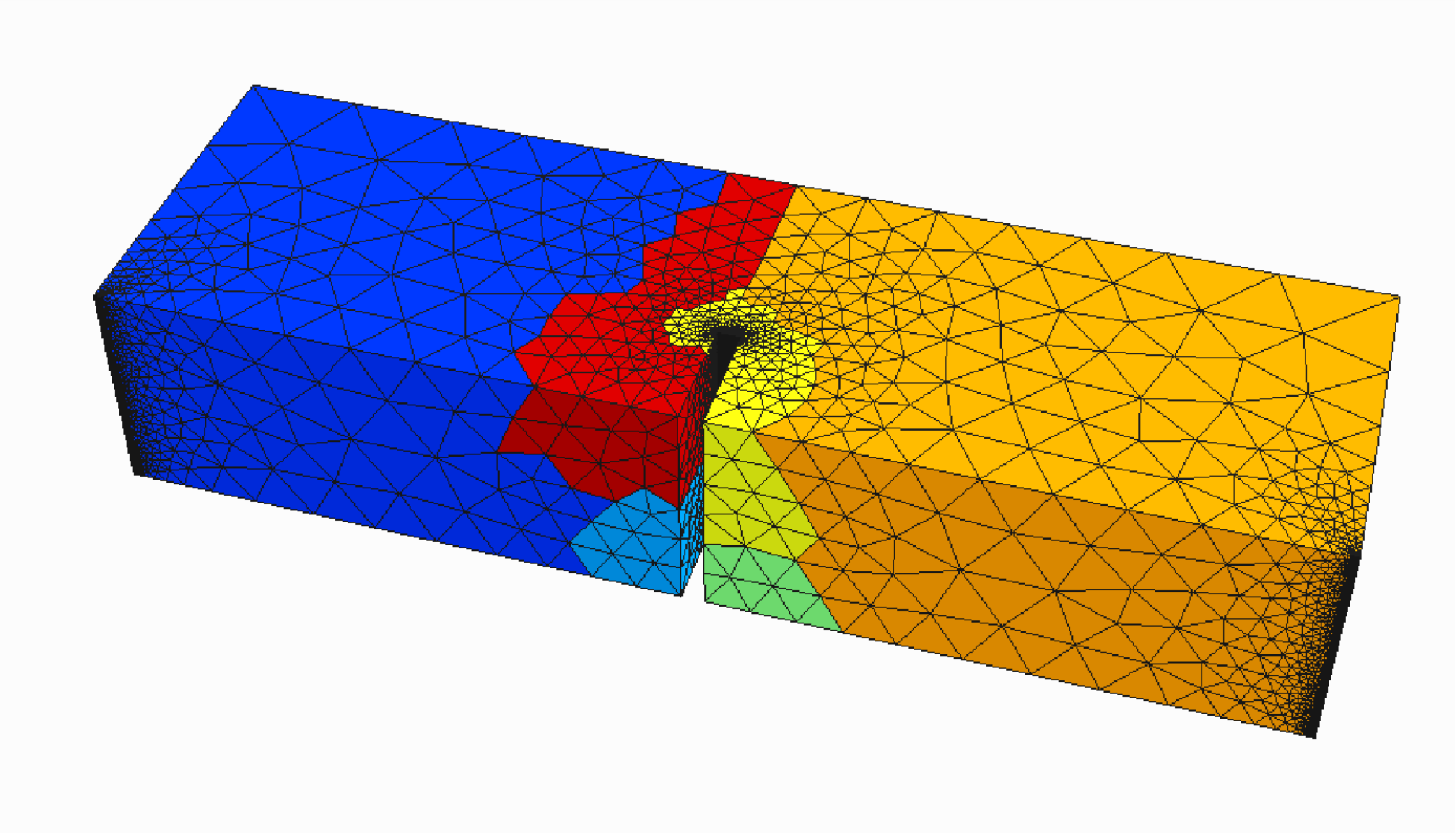} \\[0pt]
\begin{tabular}{cc}
\includegraphics[width=4cm]{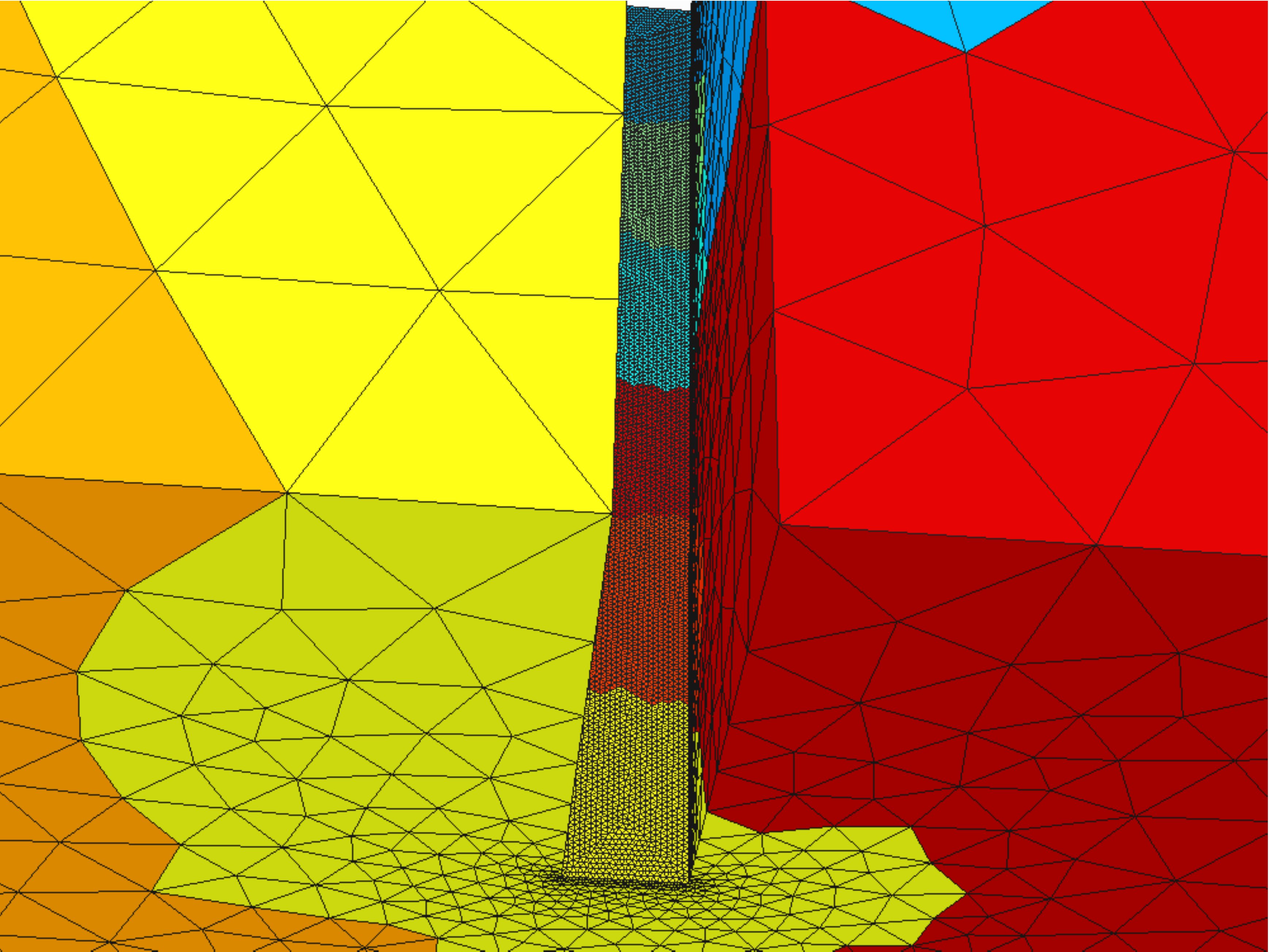} & %
\includegraphics[width=4cm]{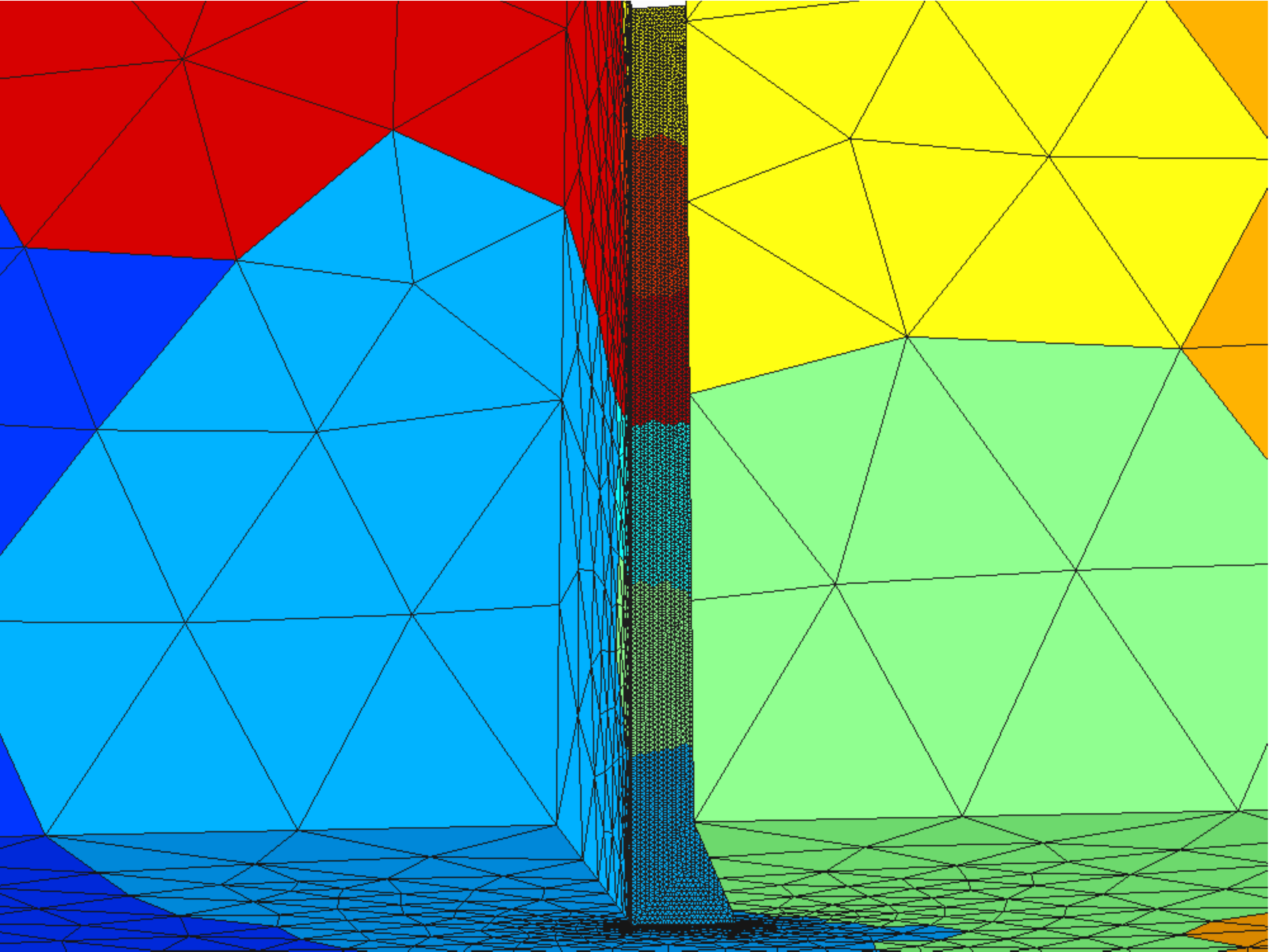}%
\end{tabular}%
\end{center}
\caption{Finite element discretization and substructuring of the beam with a
notch, consisting of $143,451$ degrees of freedom, $8$ substructures, $31$
corners, $18$ edges, and $19$ faces.}
\label{fig:beam}
\end{figure}

\begin{table}[tbp]
\centering
\begin{tabular}{c|r|r|r}
constraint & $Nc$ & $\kappa$ & $it$ \\ \hline
c & $0$ & $127.1$ & $79$ \\
c+e & $111$ & $101.0$ & $61$ \\
c+e+f & $168$ & $22.4$ & $32$ \\
c+e+f (3eigv) & $168$ & $13.2$ & $30$%
\end{tabular}%
\caption{Results for the beam with a notch. The headings are same as in
Table~\protect\ref{tab:jettube-nonadaptive}.}
\label{tab:beam-nonadaptive}
\end{table}

\begin{table}[tbp]
\centering
\begin{tabular}{r|r|r|r|r}
$\tau$ & $\widetilde{\omega}$ & $Nc$ & $\kappa$ & $it$ \\ \hline
$\infty${\small {(=c+e)}} & $149.0$ & $111$ & $101.0$ & $61$ \\
$20$ & $18.272$ & $119$ & $19.011$ & $41$\\
$10$ & $9.986$ & $134$ & $8.505$ & $31$\\
$5$ & $4.994$ & $163$ & $4.655$ & $24$\\
$3$ & $2.998$ & $215$ & $2.873$ & $18$\\
$2$ & $1.993$ & $340$ & $2.145$ & $14$
\end{tabular}%
\caption{Results for the beam with a notch. The headings are same as in
Table~\protect\ref{tab:jettube-adaptive}.}
\label{tab:beam-adaptive}
\end{table}

\begin{figure}[tbp]
\begin{center}
\includegraphics[width=9cm]{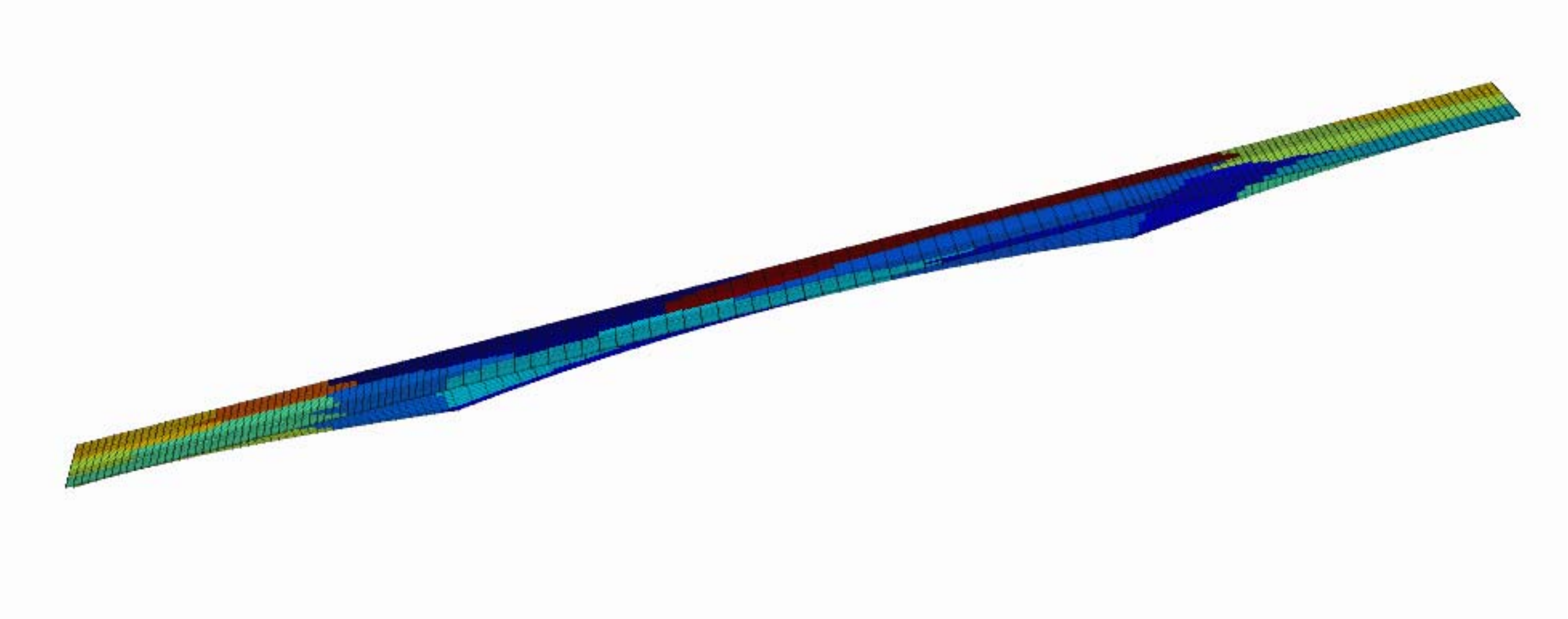} \\[0pt]
\includegraphics[width=9cm]{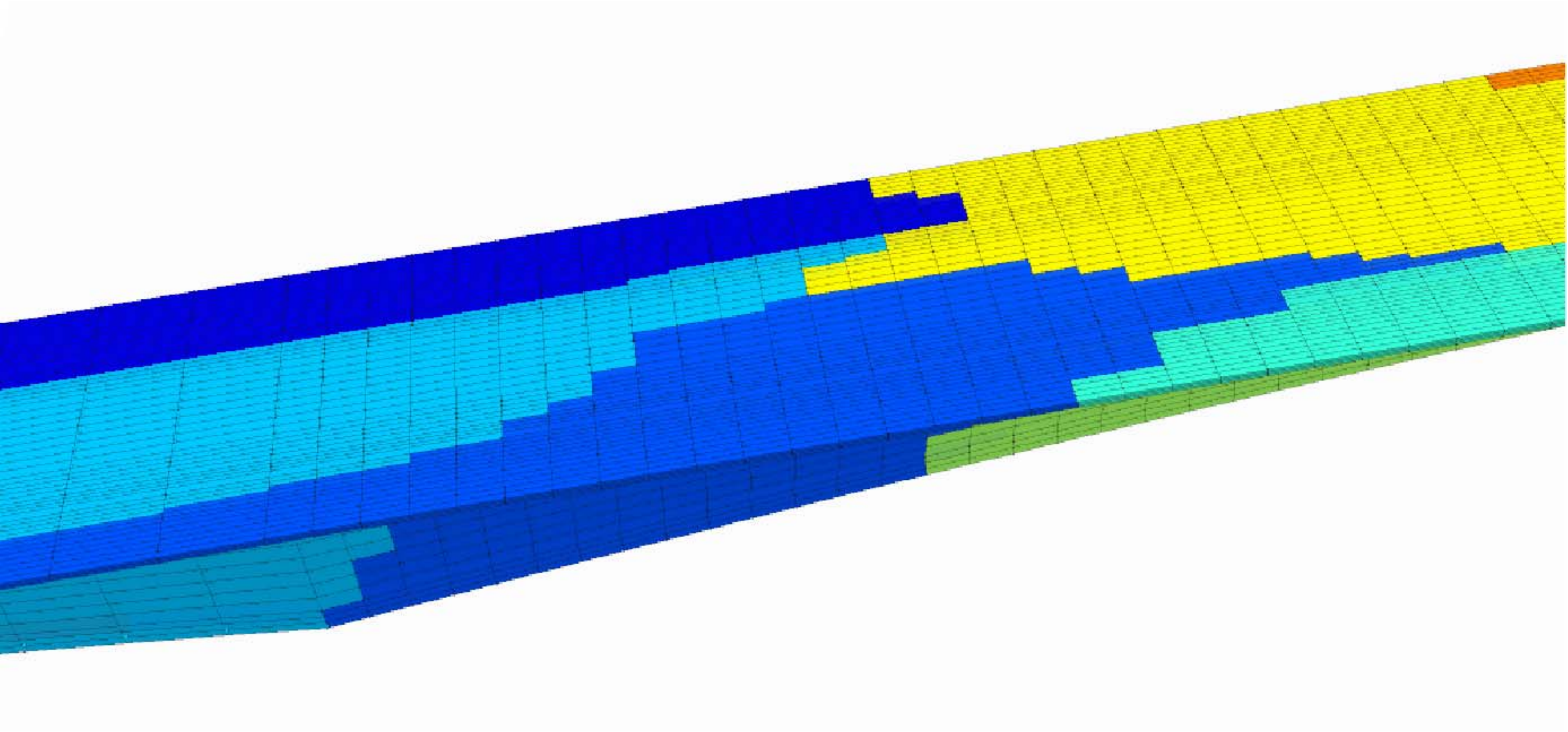} \\[0pt]
\begin{tabular}{cc}
\includegraphics[width=5cm]{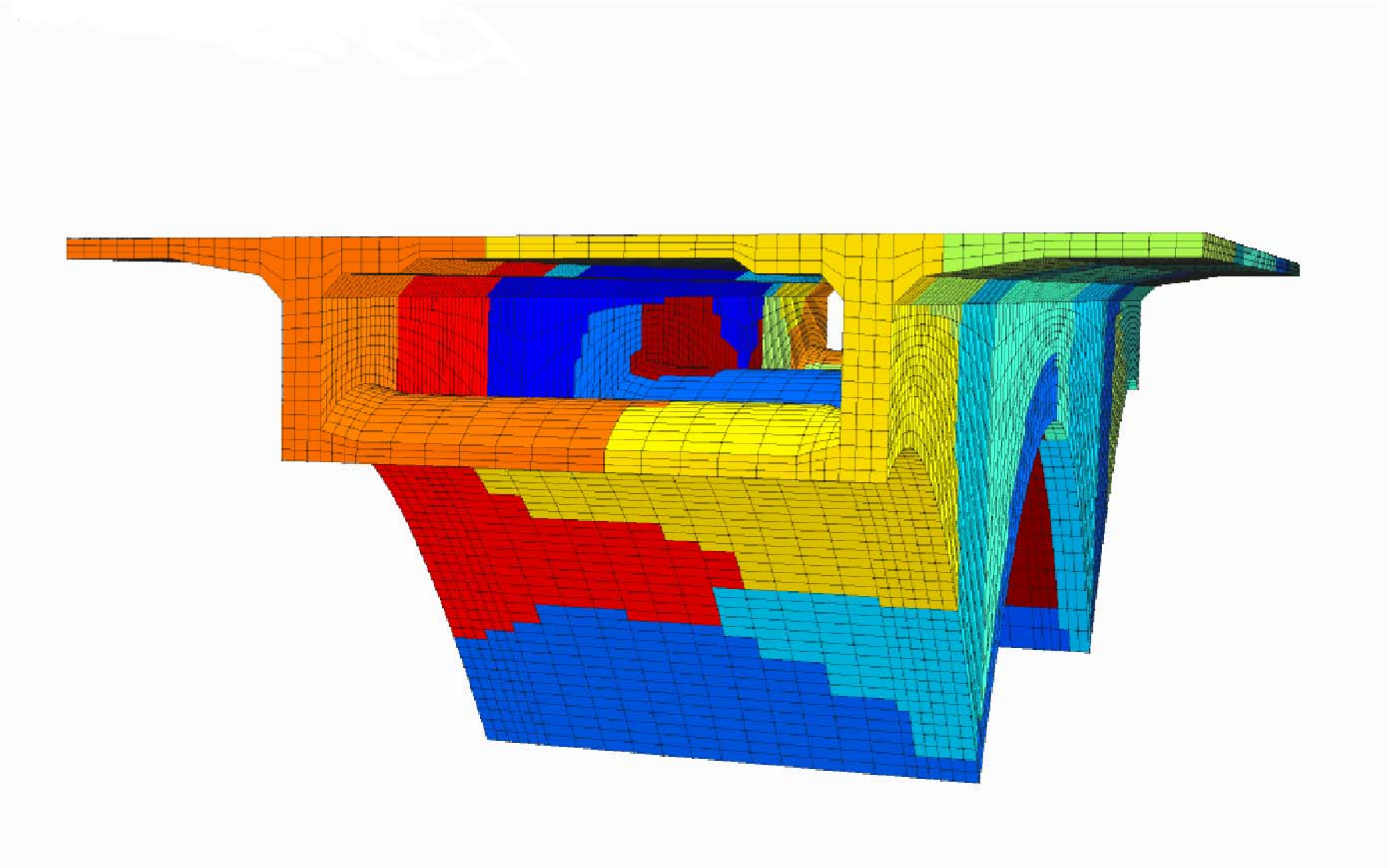} & %
\includegraphics[width=5cm]{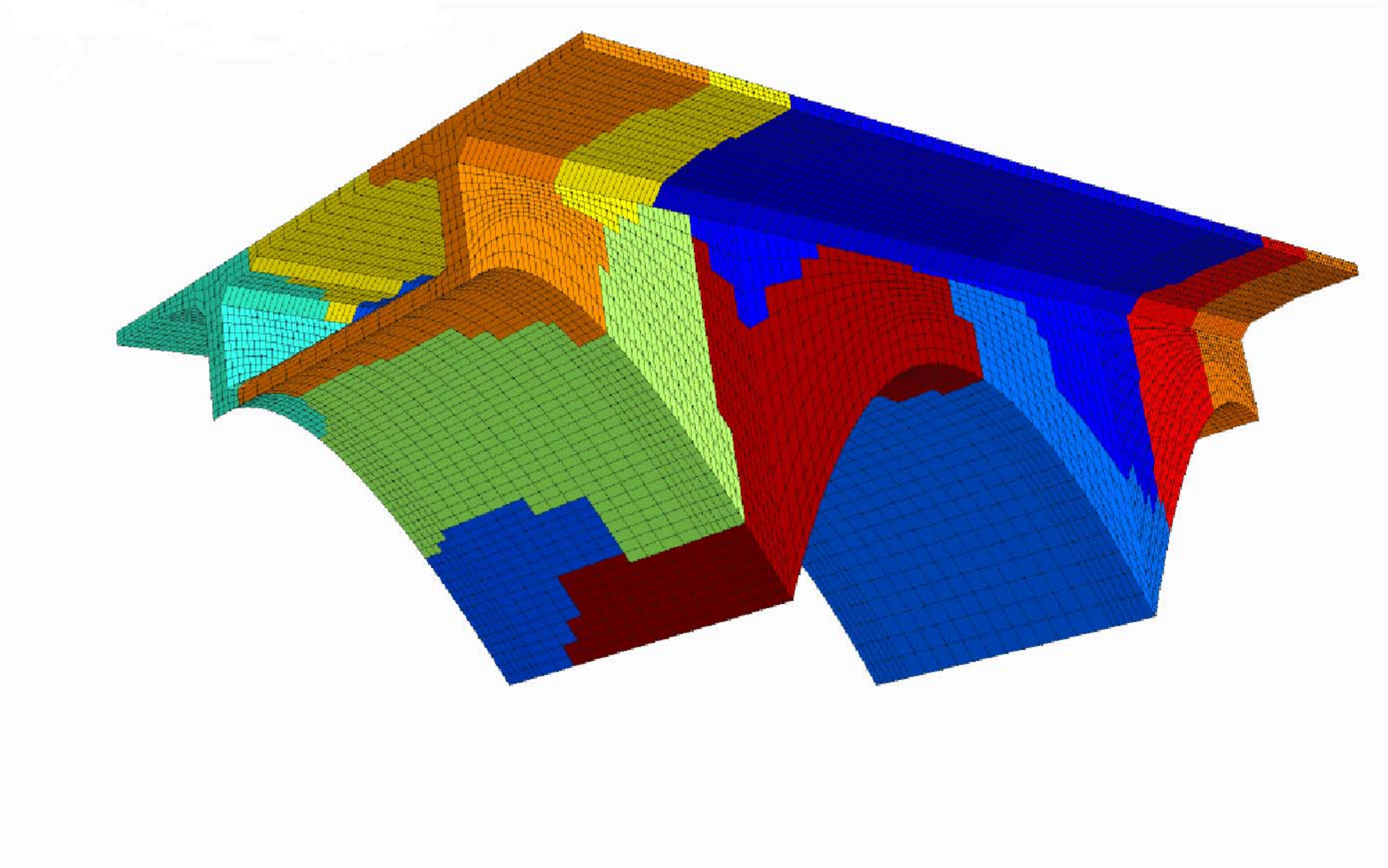}%
\end{tabular}%
\end{center}
\caption{Finite element discretization and substructuring of the bridge
construction, consisting of $157,356$ degrees of freedom, $16$
substructures, $250$ corners, $30$ edges, and $43$ faces.}
\label{fig:bridge}
\end{figure}

\begin{table}[tbp]
\centering
\begin{tabular}{c|r|r|r}
constraint & $Nc$ & $\kappa$ & $it$ \\ \hline
c & $0$ & $2301.4$ & $224$ \\
c+e & $180$ & $2252.4$ & $220$ \\
c+e+f & $309$ & $653.6$ & $160$ \\
c+e+f (3eigv) & $309$ & $177.8$ & $103$%
\end{tabular}%
\caption{Results for the bridge construction. The headings are same as in
Table~\protect\ref{tab:jettube-nonadaptive}.}
\label{tab:bridge-nonadaptive}
\end{table}

\begin{table}[tbp]
\centering
\begin{tabular}{r|r|r|r|r}
$\tau$ & $\widetilde{\omega}$ & $Nc$ & $\kappa$ & $it$ \\ \hline
$\infty${\small {(=c+e)}} & $6500.5$ & $180$ & $2252.4$ & $220$ \\
$650$  & $589.338$ & $185$ & $483.517$ & $169$ \\
 $30$   & $29.568$ & $292$ & $28.739$ & $64$ \\
    $5$  &  $4.997$ & $655$ & $5.014$ &  $26$ \\
    $2$  & $1.998$ & $1301$ & $2.011$  & $14$ \\
\end{tabular}%
\caption{Results for the bridge construction. The headings are same as in
Table~\protect\ref{tab:jettube-adaptive}.}
\label{tab:bridge-adaptive}
\end{table}

\begin{figure}[tbp]
\begin{center}
\begin{tabular}{cc}
\includegraphics[width=5cm]{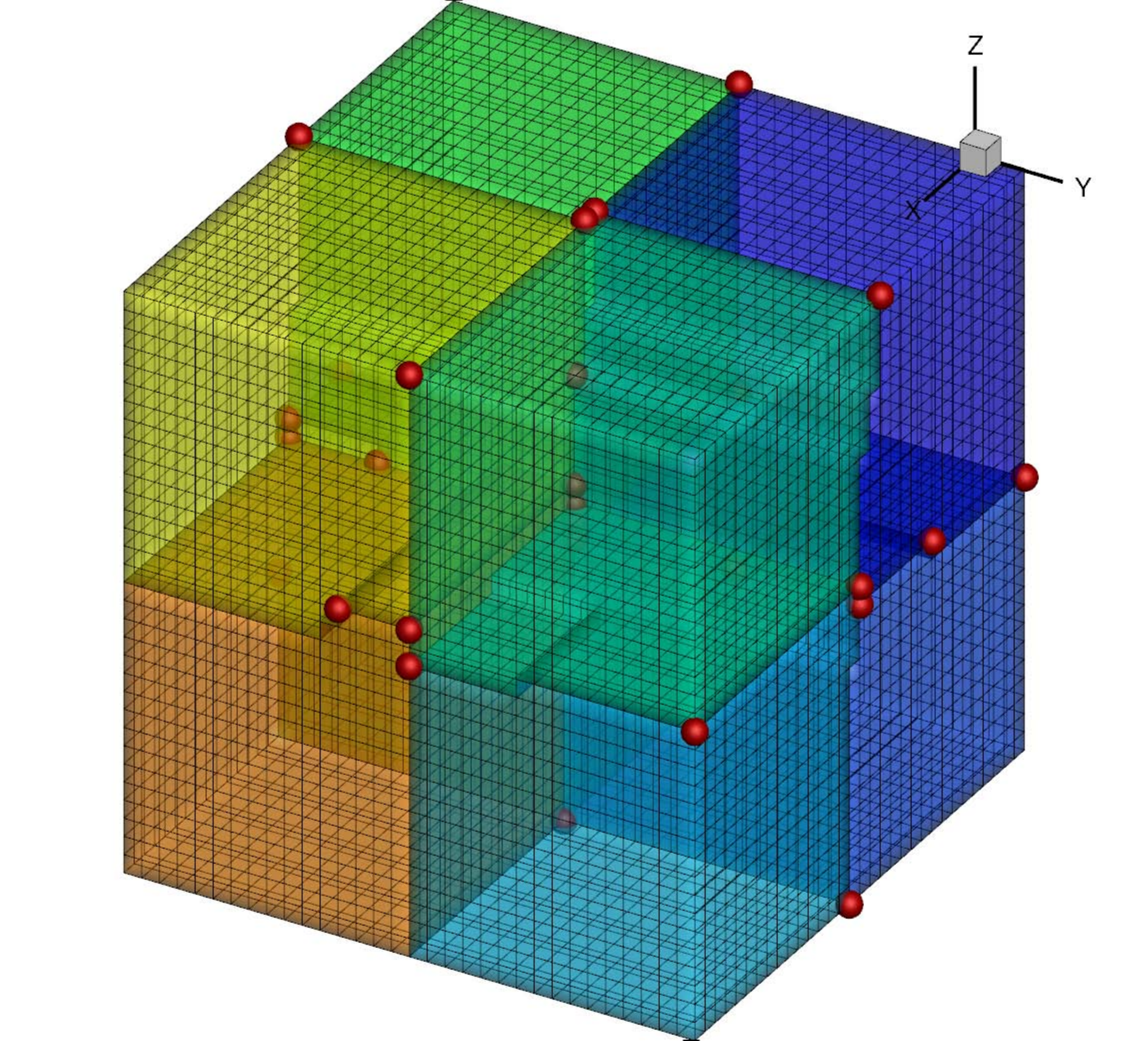} & %
\includegraphics[width=5cm]{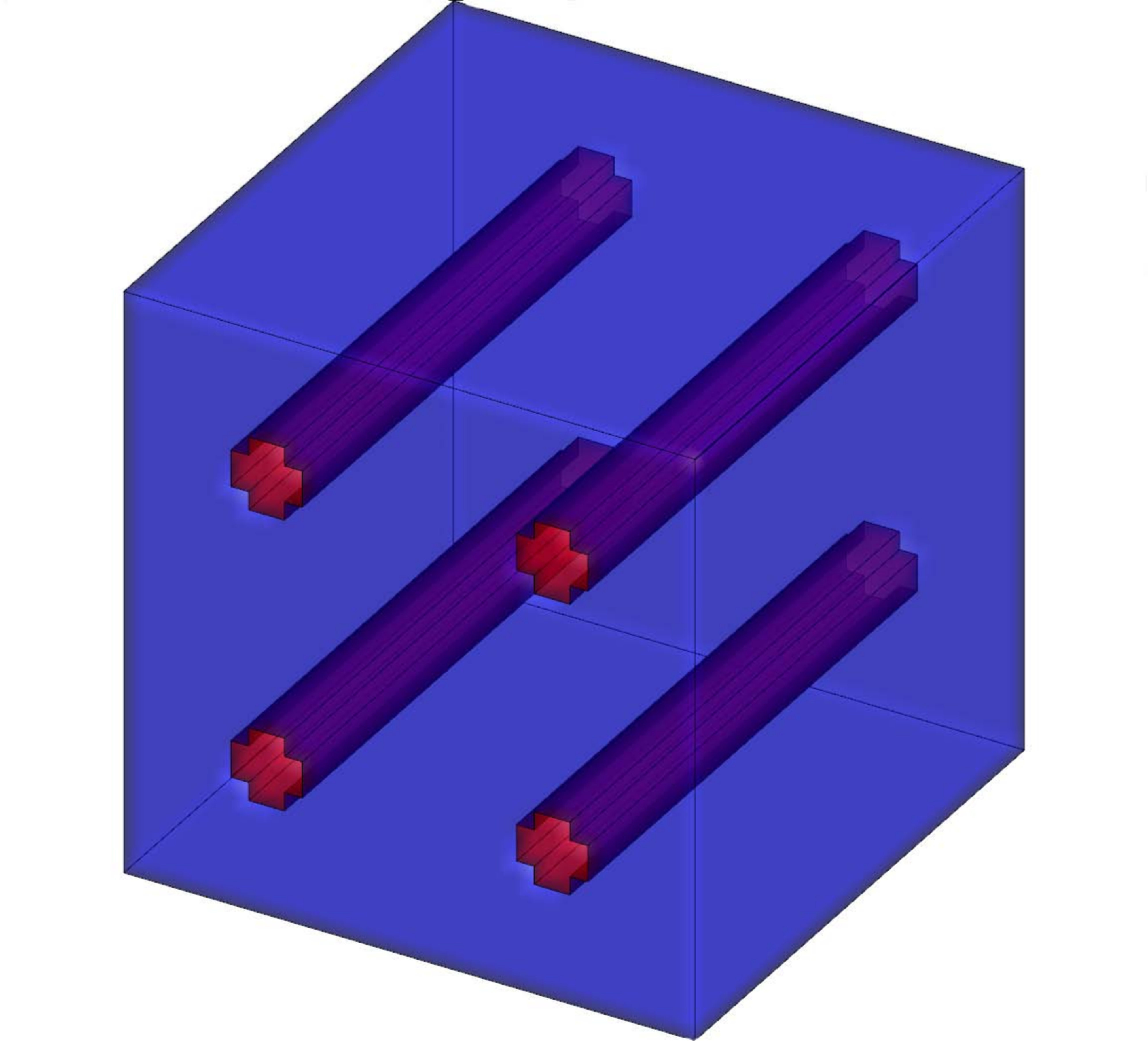}%
\end{tabular}%
\end{center}
\caption{Finite element discretization and substructuring of the cube with
jumps in coefficients, consisting of $107,811$ degrees of freedom, $8$
substructures, $30$ corners, $16$ edges, and $15$ faces.}
\label{fig:compo8}
\end{figure}

\begin{table}[tbp]
\centering
\begin{tabular}{c|r|r|r}
constraint & $Nc$ & $\kappa$ & $it$ \\ \hline
c & $0$ & $408,101.0$ & $326$ \\
c+e & $108$ & $125,390.0$ & $234$ \\
c+e+f & $153$ & $18,914.9$ & $169$ \\
c+e+f (3eigv) & $153$ & $1266.4$ & $71$%
\end{tabular}%
\caption{Results for the cube with jumps in coefficients. The headings are
same as in Table~\protect\ref{tab:jettube-nonadaptive}.}
\label{tab:compo8-nonadaptive}
\end{table}

\begin{table}[tbp]
\centering
\begin{tabular}{r|r|r|r|r}
$\tau$ & $\widetilde{\omega}$ & $Nc$ & $\kappa$ & $it$ \\ \hline
$\infty${\small {(=c+e)}} & $270,000.0$ & $108$ & $125,390.0$ & $234$ \\
$10,000$ & $5145.293$  & $118$ & $1843.35$  & $90$ \\
 $1,000 $ & $380.019$  &  $129$  & $173.562$ & $35$ \\
   $ 100$ &  $ 77.189$  & $132$  & $ 6.423$ &  $24 $  \\
     $   5 $ &  $ 4.990$ & $173$ & $4.362$  & $20$ \\
      $  2$  &  $1.998$  & $451$  &  $2.803$   & $16$
\end{tabular}%
\caption{Results for the cube with jumps in coefficients. The headings are
same as in Table~\protect\ref{tab:jettube-adaptive}.}
\label{tab:compo8-adaptive}
\end{table}

\begin{figure}[tbp]
\begin{center}
\includegraphics[width=12cm]{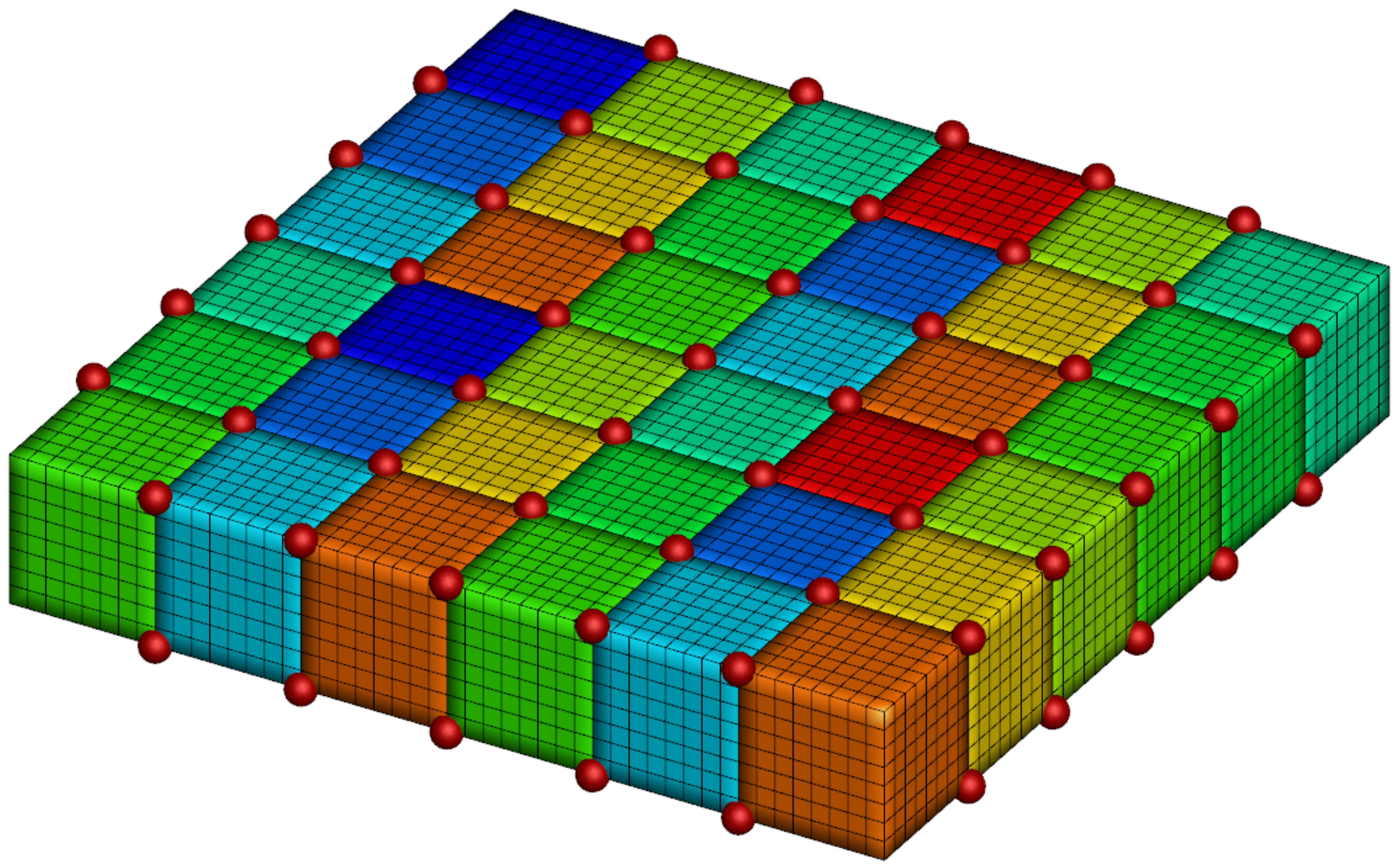}
\end{center}
\caption{Example of a configuration of planar cubes test problem with 36 
subdomains and $H/h=8$, red dots represent corners.}
\label{fig:planar_cubes}
\end{figure}

\begin{table}[tbp]
\centering
\begin{tabular}{c|c|c|c|c|c|c|c}
number of subdomains & 4 & 9 & 16 & 25 & 36 & 49 & 64 \\
degrees of freedom & 7803 & 16,875 & 29,403 & 45,387 & 64,827 & 87,723
& 114,075 \\ \hline
condition number est. & 28.3 & 38.0 & 42.2 & 44.4 & 45.7 & 46.5 & 47.1 \\
number of PCG iterations & 13 & 26 & 36 & 42 & 44 & 46 & 47 \\ \hline
analysis by MUMPS (sec) & 0.2 & 0.5 & 1 & 15 & 14 & 16 & 19 \\
factorization by MUMPS (sec) & 0.5 & 0.4 & 0.8 & 12 & 10 & 12 & 14 \\
PCG iterations (sec) & 0.8 & 3.6 & 13 & 613 & 524 & 579 & 643 \\ \hline
one PCG iteration (sec) & 0.06 & 0.14 & 0.4 & 15 & 12 & 13 & 14 \\ \hline
total wall time (sec) & 3 & 6 & 19 & 715 & 616 & 696 & 794%
\end{tabular}%
\caption{Weak scaling on planar cubes problem (e.g. Fig. \protect \ref{fig:planar_cubes}), 
corners only, $H/h=8$.}
\label{tab:scaling_cornersHh8}
\end{table}

\begin{figure}[tbp]
\begin{center}
\includegraphics[width=8cm,angle=-90]{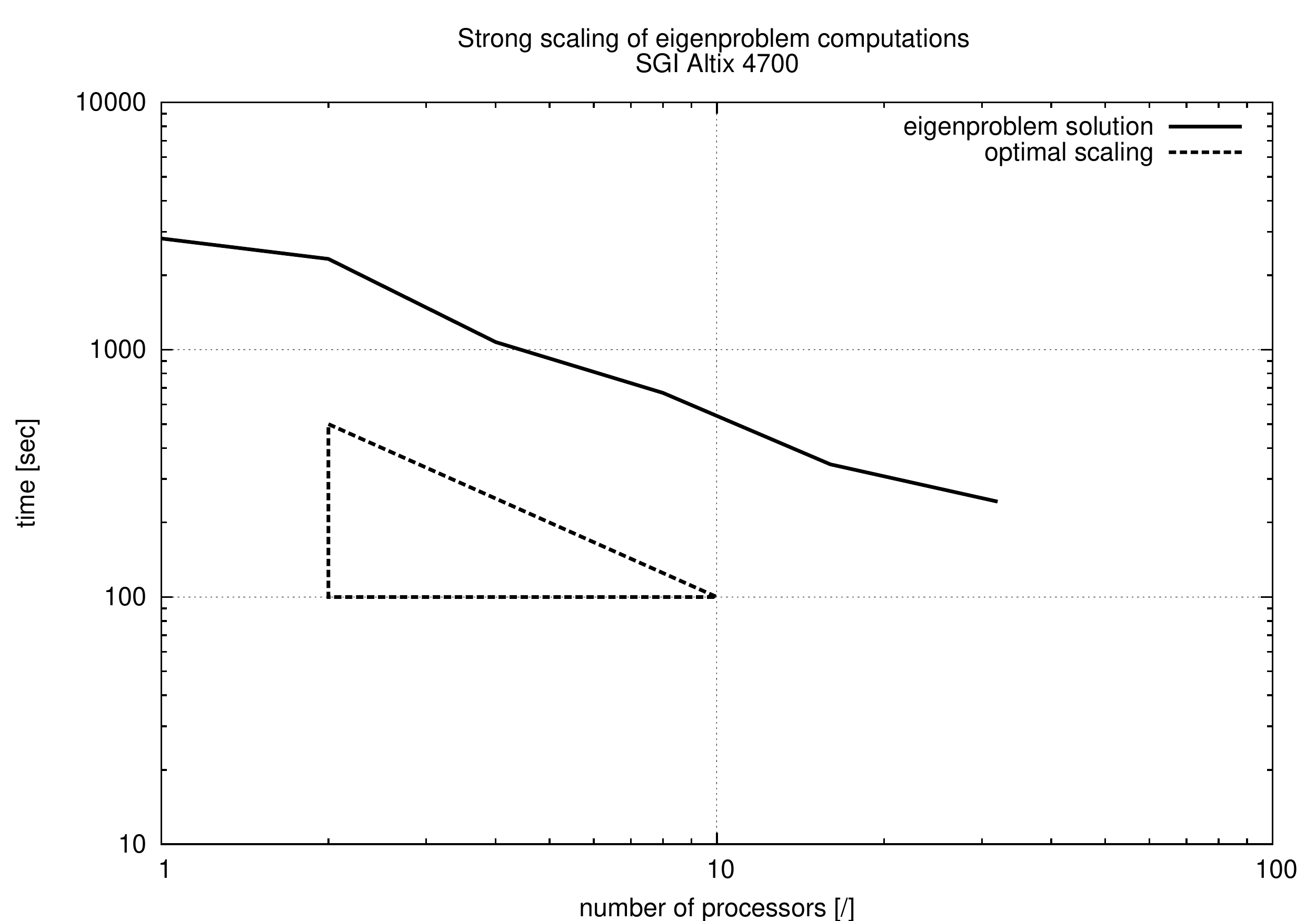}
\end{center}
\caption{Dependence of the computational time on the number of processors for solution of local eigenproblems, nozzle box problem, 16 subdomains, 30 eigenproblems.}
\label{fig:scaling_eigenproblems_select}
\end{figure}

\end{document}